\documentclass{amsart}
\usepackage[margin=1.25in,dvips]{geometry}

\usepackage{cite}
\usepackage[numbers]{natbib}
\usepackage{amssymb,mathrsfs,graphicx}
\usepackage{amsmath,amsthm}
\usepackage{bbm}
\usepackage{xcolor}
\usepackage{soul} 
\usepackage{comment}
\usepackage{cleveref}
\usepackage{enumitem}

\crefformat{equation}{(#2#1#3)}
\crefrangeformat{equation}{(#3#1#4)--(#5#2#6)}
\crefmultiformat{equation}{(#2#1#3)}%
{ and~(#2#1#3)}{, (#2#1#3)}{ and~(#2#1#3)}

\usepackage{tikz}
\usetikzlibrary{angles, quotes, decorations.pathreplacing, arrows.meta, calc}

\numberwithin{equation}{section}

\title[Long-range scattering]{Long-range scattering}
\dedicatory{Dedicated to I.M. Sigal on the event of his 80th birthday.}

 \author[Avy Soffer]{Avy Soffer}
 \address[Avy Soffer]{\newline
        Department of Mathematics, \newline
         Rutgers University, New Brunswick, NJ 08903 USA.}
  \email[]{soffer@math.rutgers.edu}
\author[Gavin Stewart]{Gavin Stewart}
\address[Gavin Stewart]{\newline
        Department of Mathematics, \newline
         Arizona State University, Tempe, AZ 85281 USA.}
  \email[]{gavin.s.stewart@asu.edu}
  
\author[Xiaoxu Wu]{Xiaoxu Wu}
\address[Xiaoxu Wu]{\newline
        Mathematical Sciences Institute, \newline
        Australia National University, Acton, ACT 2601, Australia}
 \email[]{Xiaoxu.Wu@anu.edu.au}

\newtheorem{theorem}{Theorem}[section]
\newtheorem{lemma}[theorem]{Lemma}

\newtheorem{corollary}[theorem]{Corollary}
\newtheorem{proposition}[theorem]{Proposition}
\newtheorem{remark}{Remark}[section]

\newcommand{\R}{\mathbb{R}}

\newcommand{\ext}{\textup{ext}}

\newcommand{\bbR}{\mathbb{R}}

\newcommand{\cS}{\mathcal{S}}

\newcommand{\cF}{\mathcal{F}}

\newcommand{\rmI}{\mathrm{I}}
\newcommand{\rmII}{\mathrm{II}}

\newcommand{\trmIst}{\tilde{\mathrm{I}}_\text{stat}}

\newcommand{\rmInst}{\mathrm{I}_\text{nonstat}}
\newcommand{\trmInst}{\tilde{\mathrm{I}}_\text{nonstat}}
\newcommand{\rmIInst}{\mathrm{II}_\text{nonstat}}

\newcommand{\urem}{u_\textup{rem}}
\newcommand{\uscat}{u_\textup{scat}}
\newcommand{\uwl}{u_\textup{wl}}
\newcommand{\Vint}{V_\textup{int}}
\newcommand{\Vext}{V_\textup{ext}}
\newcommand{\CVext}{\mathcal{V}_\textup{ext}}
\usepackage{xcolor}

\newcommand{\eq}{\begin{equation}}
\newcommand{\eeq}{\end{equation}}

\DeclareMathOperator*{\slim}{s-lim}
\DeclareMathOperator*{\wlim}{w-lim}
\DeclareMathOperator{\supp}{supp}

\DeclareMathOperator{\ad}{ad}

\newcommand{\jBra}[1]{\left\langle #1 \right\rangle}

\begin{document}
%%%%%%%%%%%%%%%%

\date{\today}


%\thanks{\textbf{Acknowledgment.} }

\begin{abstract}
We study the scattering problem for a long range potential, which is time dependent. We prove the existence and completeness of the scattering wave operators, and find some properties of the weakly localized, non-scattering part of the solution.
The method we use follows recent methods introduced and applied to short range systems.
%The present work can be seen as an extension of the scattering results of~\cite{SW20221} for the Schr\"odinger equation. 
\end{abstract}
\maketitle
%\centerline{\date}
%{\bf Keywords}:The Boltzmann equation, nonlinear functional, $L^1$ stability.

%\tableofcontents
\section{Introduction}

In this paper, we will study the long-time behavior of solutions to the Schr\"odinger equation
\begin{equation}\label{eqn:main}
    i\partial_t u = (-\Delta + V(x,t)) u
\end{equation}
in dimension $3$.  Here, $V(x,t)$ is a (possibly time-dependent) potential satisfying the finite symbol condition
\begin{equation}\label{eqn:V-cond}
    \sup_{x,t} \jBra{x}^{a+\mu} |D_x^a V(x,t)| \lesssim 1
\end{equation}
for some $\mu \in (1/2, 1]$ and for all $a \in [0, 8]$.  Our goal is to understand all possible asymptotic behaviors for solutions to~\eqref{eqn:main}.  

In the time-independent case $V = V(x)$, the result is completely understood: The spectral theorem for $-\Delta + V$ induces a decomposition $L^2(\bbR^3) = \mathcal{H}_d + \mathcal{H}_c$, together with a corresponding decomposition $u(t) = P_d u(t) + P_c u(t)$.  Furthermore, the spectral decomposition is reflected in the physical-space behavior of the solution: The discrete part $P_d u(t)$ remains localized near the origin for all time, while the continuous component $P_c u(t)$ exhibits modified scattering asymptotics and exits any compact set $K$ as $t \to \infty$.

In the general time-dependent case, we no longer have access to the spectral theorem, and it does not appear possible to obtain such a precise decomposition (particularly the localization and precise behavior of $P_d u(t)$) using the tools available now.  However, we will show that even in the time dependent case it is possible to obtain a decomposition $u(t) =\uscat(t) + \uwl(t)$, where $\uscat$ exhibits modified scattering asymptotics and $\uwl(t)$ is a remainder term which spreads sub-ballistically (see~\Cref{thm:main-thm,thm:better-rate} below for a precise statement).  In particular, $\uscat$ and $\uwl$ concentrate on sets whose separation grows like $t$ as $t \to \infty$ (see~\Cref{rmk:phys-space-separation}).

\subsection{Background and prior work}

The condition $\mu \leq 1$ in~\eqref{eqn:V-cond} means that the potential $V$ is long-range and induces additional non-perturbative corrections to the scattering dynamics.  In particular, the (free) wave-operator
\begin{equation*}
    \Omega_\text{free} = \slim_{t \to \infty} U(0,t) e^{it\Delta}
\end{equation*}
fails to exist in $L^2$. Here, $U(t,0)$ is the propagator for the full dynamics:
\begin{equation}\label{eqn:full-dyn-def}\left\{\begin{array}{c}
    i \partial_t U(t,s) = (-\Delta + V(x,t)) U(t,s)\\
    i \partial_s U(t,s) = -U(t,s)(-\Delta + V(x,s))\\
    U(s,s) = I
\end{array}\right.\end{equation}
To understand the asymptotic behavior of solutions, we must choose an appropriate modified scattering dynamics to capture the additional long-range corrections introduced by the slowly decaying potential.  In this paper, we will work with the Dollard dynamics~\cite{Dollard}, whose evolution is given by
\begin{equation}\label{eqn:Dollard-dyn-def}
    \left\{\begin{array}{c}
        i \partial_t U_D(t,s) = (-\Delta + V(2pt, t)) U_D(t,s)\\
        i \partial_s U_D(t,s) = -U_D(t,s) (-\Delta + V(2ps, s))\\
        U_D(s,s) = I
        \end{array}\right.
\end{equation}
where $p = -i\nabla_x$ is the momentum operator.  For potentials $V$ with $\mu \in (1/2, 1]$, the Dollard dynamics captures the additional corrections to the scattering dynamics, allowing us to construct the modified wave operator
\begin{equation}
    \Omega_D = \slim_{t \to \infty} U(0,t) U_D(t,0)
\end{equation}
as a strong $L^2$ limit (see~\Cref{cor:full-WO}).  

We remark that there are a number of other methods to construct modified wave operators. The Buslaev-Matveev dynamics~\cite{buslaev1970wave} incorporates higher-order corrections to the Dollard dynamics, which allows the construction of modified wave operators for potentials $V$ satisfying versions of~\eqref{eqn:V-cond} with $\mu \in (0, 1/2]$, although the corrections become increasingly difficult to compute as $\mu \to 0$.  Another approach, which was first introduced by~\cite{Hormander}, is to define dynamics using exact trajectories for the corresponding classical Hamiltonian~\cite{sigal1990long}.  This approach has the advantage of applying to potentials with very slow decay rates, although the resulting dynamics are not as explicit as~\eqref{eqn:Dollard-dyn-def}.  At least in the time-independent case, it is also possible to construct the modified wave operator directly, either by using the resolvent formalism~\cite{ikebe1982stationary,saito2006spectral}, by using solutions to the classical eikonal equation to define the wave operator as a Fourier integral operator~\cite{isozaki1985modified}.  There are also methods to obtain the scattering matrix algebraically using asymptotic observables~\cite{Thirring}. Thirring's argument works for Coulomb Scattering, due to the extra symmetries of this potential. We also remark that these techniques can be generalized to treat $N$-body quantum systems: See~\cite{Merkuriev,skibsted2023stationary}.  For a more detailed treatment of these alternative modified dynamics, we refer the reader to~\cite{derezinskiScatteringTheoryClassical1997}.  Among all of these dynamics, a key advantage of the Dollard dynamics is their simplicity:
$U_D(t,s)$ is an explicit Fourier integral operator:
\begin{equation}\label{eqn:Dollard-Fourier-rep}
    \cF U_D(t,s) g = \exp\left(-i(t-s) |\xi|^2 -i \int_s^t V(2\xi \tau, \tau)\;d\tau\right) \hat{g}(\xi).
\end{equation}
This allows us to derive dispersive estimates for $U_D$ from stationary phase methods.

The existence of $\Omega_D$ shows that there are solutions to~\eqref{eqn:main} which exhibit Dollard dynamics; that is, given a particular modified scattering trajectory $U_D(t,0) u_+$, the solution $u(t)$ to~\eqref{eqn:main} with initial data $u(0) = \Omega_D \uscat(0)$ is well approximated by $\uscat(t)$ in $L^2$ as $t \to \infty$:
\begin{equation*}
    \lim_{t \to \infty} \lVert u(t) - U_D(t,0) u_+ \rVert_{L^2} = 0
\end{equation*}
Based on physical considerations, we would prefer to have an \textit{asymptotic completeness} result enumerating all the possible asymptotic behaviors of solutions to~\eqref{eqn:main}.  In the case where $V$ is time-independent, we know that any solution $u(t)$ to~\eqref{eqn:main} decomposes as $t \to \infty$ as
\begin{equation*}
    u(t) = U_D(t,0) u_+ + u_b(t) + o_{t \to \infty}(1)
\end{equation*}
where $u_b(t)$ is a time-quasiperiodic sum of eigenfunctions of $-\Delta + V(x)$, and $o_{t \to \infty}(1)$ represents a term with vanishing $L^2$ norm as $t \to \infty$.  In the case where there are no bound states, asymptotic completeness is equivalent to showing that $\Omega_D$ is surjective onto $L^2$, or equivalently that
\begin{equation}\label{eqn:Omega-D-adj-slim}
    \Omega_D^* = \slim_{t \to \infty} U_D(0,t) U(t,0)
\end{equation}
exists as an $L^2$-strong limit.\footnote{The existence of $\Omega_D$ guarantees that $\Omega_D^*$ can be written as an $L^2$-\textit{weak} limit: $\Omega_D^* = \wlim_{t \to \infty} U_D(0,t) U(t,0)$.}  If there are bound states, the spectral theorem allows us to obtain asymptotic completeness provided that $\Omega_D$ is surjective onto the continuous spectral subspace, which again follows from~\eqref{eqn:Omega-D-adj-slim}.

In the case where $V$ is time-dependent, we can no longer appeal to spectral theory (except in the special case where $V(x,t)$ is periodic or quasiperiodic in time, where we can appeal to Floquet theory~\cite{yajimaScatteringTheorySchrodinger1977,howlandScatteringTheoryHamiltonians1979}; see also~\cite{howlandStationaryScatteringTheory1974,JL1991,soffer2025local}), so showing~\eqref{eqn:Omega-D-adj-slim} is no longer sufficient to prove asymptotic completeness.  Indeed, this problem is already present in the $N$-body scattering problem with time-independent potentials, and attempts to address it led to the development of propagation estimates and Mourre theory by Mourre \cite{Mourre}, Enss \cite{Enss}, and Sigal-Soffer \cite{SigalSoffer}, and ultimately to the resolution of the problem by using microlocalization with asymptotic observables in~\cite{SigalSoffer,derezinskiScatteringTheoryClassical1997,hunziker2000time,skibsted2023stationary}.

In the general time-dependent case, asymptotic completeness appears to be beyond the reach of modern tools.  To prove asymptotic completeness for the time-independent $2$-body problem, one first projects onto to the continuous spectrum, and then argues using local decay and dispersive estimates that the effect of the potential is perturbative on the continuous spectrum.  At a very high level, the argument for the time-independent $N$-body problem is similar: One begins by considering how particles arrange themselves into clusters, and then argues that intracluster dynamics are described by the discrete spectrum of the corresponding intercluster Schr\"odinger operator, while inter-cluster interactions can be treated perturbatively due to the spatial decay of the potential.  In the time-dependent setting modeled by~\eqref{eqn:main}, the potential can still be treated perturbatively far from the origin, but we can no longer appeal to the spectral theorem to get asymptotics in the region where the potential dominates the dynamics.

To get around this problem, we will follow the approach first introduced by Tao in~\cite{taoAsymptoticBehaviorLarge2004b} and refined in later works~\cite{sofferScatteringLocalizedStates2024a,Soffer-Wu,taoConcentrationCompactAttractor2007,taoGlobalCompactAttractor2008b,stewartWeaklyLocalizedStates2025,soffer2025local,liu2025large,soffer2025Lp}. The approach begins by proving the preliminary decomposition
\begin{equation}\label{eqn:prelim-decomp-intro}
    u(t) = U_D(t,0) u_+ + \urem(t)
\end{equation}
where the remainder term $\urem(t)$ is \textit{weakly bound} in the sense that
\begin{equation}\label{eqn:weak-loc-intro}
    \wlim_{t \to \infty} U_D(0,t) \urem(t) = 0
\end{equation}
in $L^2$.  In particular, this is equivalent to showing the existence of the $L^2$ weak limit
\begin{equation*}
    \wlim_{t \to \infty} U_D(0,t) u(t) = u_+
\end{equation*}
Although precise asymptotics for $\urem(t)$ appear to be out of reach, one can still prove that $\urem(t)$ has localization properties.  In high dimensions, it is possible to prove that $\urem(t)$ is localized uniformly in time~\cite{taoGlobalCompactAttractor2008b,taoConcentrationCompactAttractor2007,soffer2025local}.  Due to the weaker dispersion in dimension $3$, it does not appear possible to prove precise localization; however, by adapting the arguments introduced in~\cite{liu2025large,sofferScatteringLocalizedStates2024a}, we are able to prove that $\urem(t)$ spreads sublinearly in time:
\begin{equation*}
    \lim_{\epsilon \to 0} \lim_{t \to \infty} \lVert \urem(t) \rVert_{L^2(|x| > \epsilon t)} = 0
\end{equation*}
Since maximal velocity estimates imply that the radiation component $U_D(t,0) u_+$ spreads linearly in time, this implies that the two components $U_D(t,0) u_+$ and $\urem(t)$ in the decomposition~\eqref{eqn:prelim-decomp-intro} become separated in physical space.

\subsection{Main Result}

We will prove the following result:
\begin{theorem}\label{thm:main-thm}
    If $\mu \in (1/2,1]$, then for any $u_0 \in L^2(\bbR^3)$, the solution $u(t)$ to~\eqref{eqn:main} decomposes as
    \begin{equation}\label{eqn:main-decomp}
        u(t) = U_D(t,0) u_+ + \urem(t)
    \end{equation}
    for some $u_+ \in L^2$ and $\urem \in L^\infty_t L^2_x$ satisfying the condition
    \begin{equation}\label{eqn:remainder-weak-lim}
        \wlim_{t \to \infty} U_D(0,t) \urem(t) = 0
    \end{equation}
    in $L^2$.  Moreover, if the solution $u(t)$ has bounded energy in the sense that
    \begin{equation}\label{eqn:energy-bound}
        \lVert u \rVert_{L^\infty_t H^1_x} < \infty
    \end{equation}
    then we further have that $u_+ \in H^1$, $\urem \in L^\infty_t H^1_x$, and for any $\epsilon > 0$,
    \begin{equation}\label{eqn:urem-slow-spreading}
        \limsup_{t \to \infty} \left\lVert F\left(\frac{|x|}{\epsilon t} \geq 1 \right) \urem(t) \right\rVert_{L^2_x}  = 0
    \end{equation}
\end{theorem}

\begin{remark}\label{rmk:phys-space-separation}
    A direct stationary phase argument shows that $U_D(t,0) u_+$ obeys the minimal velocity estimate
    \begin{equation*}
        \lim_{\epsilon \to 0} \limsup_{t \to \infty}\left\lVert F\left(\frac{|x|}{\epsilon t} \leq 1\right) U_D(t,0) u_+\right\rVert_{L^2_x}  = 0.
    \end{equation*}
    (see also~\cite{HSS} for discussions of minimal velocity estimates in more general settings). Combined with~\eqref{eqn:urem-slow-spreading}, this shows that the components $U_D(t,0) u_+$ and $\urem(t)$ in~\eqref{eqn:main-decomp} are separated in physical space.
\end{remark}

In the decomposition~\eqref{eqn:main-decomp}, neither $U_D(t,0) u_+$ nor $\urem(t)$ is a solution to~\eqref{eqn:main}.  However, it is possible to find solutions $\uscat(t)$ and $\uwl(t)$ such that
\begin{equation}\label{eqn:U_Du_+-uscat-equiv-intro}
    \lim_{t \to \infty} \lVert U_D(t,0) u_+ - \uscat(t) \rVert_{L^2_x} = 0
\end{equation}
and
\begin{equation}\label{eqn:urem-uwl-equiv-intro}
    \lim_{t \to \infty} \lVert \urem(t) - \uwl(t) \rVert_{L^2_x} = 0
\end{equation}
see~\Cref{cor:equiv-decomps}.  In particular, when proving~\eqref{eqn:urem-slow-spreading} we will in fact work with $\uwl$, since the fact that $\uwl$ satisfies~\eqref{eqn:main} allows us to use propagation estimates.  

More generally, we follow the definition introduced by Tao in~\cite{taoAsymptoticBehaviorLarge2004b} and say that a solution $\uwl(t)$ to~\eqref{eqn:main} is \textit{weakly bound} if
\begin{equation}\label{eqn:weakly-bound}
    \wlim_{t \to \infty} U_D(0,t) \uwl(t) = 0
\end{equation}
in the $L^2$ sense.\footnote{To understand the terminology, observe that for the $2$-body problem with a time-independent potential, the bound state solutions are precisely those belonging to the discrete spectral subspace of $-\Delta + V$, and the RAGE theorem~\cite{teschlMathematicalMethodsQuantum2009} guarantees that these are exactly the states satisfying~\eqref{eqn:weakly-bound}.}  In particular, the result~\eqref{eqn:urem-slow-spreading} follows immediately from the corresponding result
\begin{equation}\label{eqn:uwl-slow-spreading-intro}
        \limsup_{t \to \infty} \left\lVert F\left(\frac{|x|}{\epsilon t} \geq 1 \right) \uwl(t) \right\rVert_{L^2_x}  = 0
    \end{equation}
for weakly bound states.

The sublinear spreading bound~\eqref{eqn:uwl-slow-spreading-intro} only uses the fact that $\uwl(0)$ decays as $|x| \to \infty$ in an $L^2$ sense.  If we assume that a weakly bound state has stronger localization properties at $t = 0$, then it is possible to improve the sublinear spreading estimate~\eqref{eqn:urem-slow-spreading} to a polynomial rate:

\begin{theorem}\label{thm:better-rate}
    If $\uwl(x,t)$ is a solution to~\eqref{eqn:main} such that
    \begin{equation*}
        \wlim_{t \to \infty} U_D(0,t) \uwl(t) = 0
    \end{equation*}
    and $|x|^{1/2} \uwl(x,0) \in L^2$, then
    \begin{equation}\label{eqn:uwl-better-rate}
        \lVert |x|^{1/2} \uwl(x,t) \rVert_{L^2_x}^2 \leq \lVert |x|^{1/2} \uwl(x,0) \rVert_{L^2_x}^2 + O_\alpha\left(t^\alpha\right)
    \end{equation}
    for any $\alpha > \frac{3}{2 + 2\mu}$.
\end{theorem}
\begin{remark}
    In the case $\mu = 2$, the results of~\cite{Soffer-Wu} imply that it is possible to construct slowly spreading states (at least in high dimension).
\end{remark}

\begin{remark}\label{rmk:rate-improvements}
    We expect that localized weakly bound states in fact spread at a maximum rate of $t^{\alpha_0}$ where $\alpha_0 = \frac{2}{2+\mu}$ is the rate of spreading for an unbounded solution of the Hamiltonian ODE
    \begin{equation*}
        \ddot{x} = -\nabla_x |x|^{-\mu} 
    \end{equation*}
    Indeed,~\Cref{thm:optimal-spread-theorem} would imply~\eqref{eqn:uwl-better-rate} with $\alpha = \alpha_0$, provided that we could prove that
    \begin{equation*}
        \lim_{t \to \infty} \langle F_{\alpha_0} \gamma F_{\alpha_0} \rangle_t = 0.
    \end{equation*}
    However, actually proving such a bound appears quite delicate, and would likely require us to choose a more accurate scattering trajectory than the Dollard one.  Indeed, the requirement $\alpha > \frac{3}{2+2\mu}$ in~\Cref{thm:better-rate} is forced on us by the requirement in~\Cref{prop:ESG} that the trajectories under the exterior potential $\Vext(x,t)$ behave like those of the free problem and the control of the leading order error term~\eqref{eqn:HO-Dollard-lin-term} from the Dollard dynamics.  

    Replacing the Dollard dynamics with the $n$th order Buslaev-Matveev dynamics should replace the error term $t^{\rho - (1+\mu)\beta}$ from~\eqref{eqn:HO-Dollard-lin-term} with the error $t^{n\rho - (n+\mu)\beta}$.  Since~\Cref{prop:ESG} only involves the flow under the exterior potential $\Vext$, the restriction $\rho > 2 - (1+\mu)\beta$ does not improve under the Buslaev-Matveev dynamics, so under the $n$th order Buslaev-Matveev dynamics we expect that~\Cref{thm:better-rate} should hold for all $\alpha > \alpha_n = \frac{2n+1}{2n+(n+1)\mu}$ (in particular, the Buslaev-Matveev dynamics with $n = 2$ exactly corrects the error term~\eqref{eqn:HO-Dollard-lin-term} appearing in the Dollard dynamics, leaving only the next order error~\eqref{eqn:HO-Dollard-quad-term}).  As $n \to \infty$, $\alpha_n \downarrow \alpha_0$, so even Buslaev-Matveev corrections to the asymptotic dynamics of arbitrarily high order may not be sufficient to obtain the optimal rate of spreading.
\end{remark}

\subsection{Sketch of the proof}

Our proof of~\Cref{thm:main-thm,thm:better-rate} consists of three steps: 
\subsubsection{Step 1. Obtain the preliminary decomposition} 

First, in~\Cref{sec:Dollard-WO}, we prove the existence of the decomposition~\eqref{eqn:main-decomp}, and show that it is equivalent to a decomposition
\begin{equation}\label{eqn:sketch-decomp}
    u(t) = \uscat(t) + \uwl(t)
\end{equation}
where $\uscat(t)$ and $\uwl(t)$ are solutions to~\eqref{eqn:main}, where $\uscat(t) - U_D(t,0) u_+ \to 0$ in $L^2$ and $\uwl(t)$ is weakly bound (i.e., it satisfies~\eqref{eqn:U_Du_+-uscat-equiv-intro} and~\eqref{eqn:weak-loc-intro}).  In particular, the decomposition~\eqref{eqn:sketch-decomp} shows that the result~\eqref{eqn:uwl-sublin-spread} about $\urem$ will follow if we can show that
\begin{equation}\label{eqn:uwl-sublin-sketch}
    \lim_{\epsilon \to 0} \limsup_{t \to \infty} \left\lVert F\left(\frac{|x|}{\epsilon t} \geq 1 \right) \uwl(t) \right\rVert_{L^2_x}  = 0
\end{equation}
for any weakly bound state $\uwl(t)$.
\subsubsection{Step 2. Prove exterior radial momentum vanishes for weakly bound states}
Next, we turn our attention to~\Cref{thm:better-rate} in~\Cref{sec:slow-spread}.  In particular, at this stage we will assume that $\uwl(t)$ is initially strongly localized at $t = 1$, in the sense that $\jBra{x}^N \uwl(t) \in L^2_x$ with $N$ large.  To prove~\Cref{thm:better-rate}, we observe that the Heisenberg derivative of $|x|$ under the Schr\"odinger flow~\eqref{eqn:main} is
\begin{equation*}
    D_H |x| = 2\gamma
\end{equation*}
where
\begin{equation*}
    \gamma = \frac{1}{2}\left(p \cdot \frac{x}{|x|} + \frac{x}{|x|} \cdot p \right)
\end{equation*}
is the Morawetz vector field, which measures the momentum of the solution in the radial direction.  If we assume that the radial momentum has a limit as $t \to \infty$, then sublinear spreading of weakly bound states would be equivalent to the statement that
\begin{equation}\label{eqn:gamma-strong-vanishing}
    \langle \uwl(t), \gamma \uwl(t) \rangle =: \langle \gamma \rangle_t \to 0
\end{equation}
Indeed, for Schr\"odinger equations with time-independent potentials, $\langle \gamma \rangle_t$ vanishes for bound states.  However, in the case when $V$ is time-dependent, it is not clear whether $\lim_{t \to \infty} \langle \gamma \rangle_t$ should exist: For example, one might imagine that a carefully chosen time-dependent potential could produce a breather-like solution, with a time-periodic $\langle \gamma \rangle_t$.  However, this problem is localized near the origin: If we insert a cutoff $F_\alpha$ to $|x| \gtrsim t^\alpha$, we can prove that
\begin{equation}\label{eqn:gamma-vanishing-exterior-sketch}
    \lim_{t \to \infty} \langle F_\alpha \gamma F_\alpha \rangle_t = 0
\end{equation}
For radial solutions $\uwl$ and potentials $V$ with sufficiently rapid decay,~\eqref{eqn:gamma-vanishing-exterior-sketch} was proved in~\cite{liu2025large}.  However, the method of proof does not generalize to our setting, so we must give a new proof based on the existence of certain wave operators.  Once we know~\eqref{eqn:gamma-vanishing-exterior-sketch}, it is possible to use propagation estimates to obtain a rate of vanishing for the exterior radial momentum:
\begin{equation*}
    \langle F_\alpha \gamma F_\alpha \rangle_t \lesssim t^{\alpha-1}
\end{equation*}
which then implies that
\begin{equation*}
    \langle F_\alpha |x| F_\alpha \rangle_t \lesssim \jBra{t}^\alpha
\end{equation*}
Since $\lVert (1-F_\alpha) |x| \rVert_{L^2 \to L^2} \lesssim t^\alpha$, this proves~\Cref{thm:better-rate}.

\subsubsection{Step 3. Remove the moment condition}

It only remains to prove that~\eqref{eqn:uwl-sublin-sketch} holds under the weaker condition $\uwl(t) \in L^2$ (with no moments).  Under this condition, a density argument still gives us that
\begin{equation*}
    \langle F_\alpha \gamma F_\alpha \rangle_t \to 0
\end{equation*}
However, $\langle F_\alpha |x| F_\alpha \rangle_t$ may be infinite, so we must proceed more carefully.

To get around the unboundedness of the space weight $|x|$, we introduce the weight $G_R$, which behaves like a smoothed version of $\max(|x|, R)$.  We then prove that
\begin{equation*}
    \langle F_\alpha G_R F_\alpha \rangle_t = o(R + t)
\end{equation*}
in the joint limit $R \to \infty$, $t \to \infty$.  In particular, taking $R = \epsilon t$, we conclude that
\begin{equation*}
    \langle F_\alpha G_{\epsilon t} F_\alpha \rangle_t = o(t)
\end{equation*}
By Markov's inequality, this implies that
\begin{equation*}
    \lim_{t \to \infty} \left\lVert F\left( \frac{|x|}{\epsilon t} \geq 1\right) F_\alpha \uwl \right\rVert_{L^2_x} = 0
\end{equation*}
Since 
$$F\left( \frac{|x|}{\epsilon t} \geq 1\right) F_\alpha = F\left( \frac{|x|}{\epsilon t} \geq 1\right)$$
for $t$ sufficiently large (depending on $\epsilon$), this proves~\eqref{eqn:uwl-sublin-sketch} and completes the proof of~\Cref{thm:main-thm}.

\section{Notation and preliminary results}

\subsection{Notation}

Throughout the paper, we will work in $\bbR^3$ unless stated otherwise.  We write $\jBra{x} := (1 + |x|^2)^{1/2}$ for the Japanese bracket.

We will use tensor notation.  Given two vectors $u$ and $v$, we denote by $u \otimes v$ the matrix
\begin{equation*}
    u \otimes v = u v^T
\end{equation*}
which has the property that
\begin{equation*}
    (u \otimes v)w = (v \cdot w) u
\end{equation*}
We also use $D_x^a f$ to denote the tensor of all partial derivatives of $f$ of order $a$, and use $v^{\otimes a} = \underbrace{v \otimes v \otimes \cdots \otimes v}_{a\textup{ times}}$.  The contraction of two tensors $A$ and $B$ of the same order is denoted by $A \cdot B$; so, for example, we write the Taylor expansion of a smooth function $f$ as
\begin{equation*}
    f(x) = f(a) + D_x f(a) \cdot (x-a) + \frac{1}{2} D_x^2 f(a) \cdot (x-a)^{\otimes 2} + \cdots
\end{equation*}

\subsubsection{Inequalities and (implicit) constants}
The constants in inequalities of the form $A \lesssim B$ may change from line to line and depend only on the value of $\mu$, etc. specified by the hypothesis~\eqref{eqn:V-cond} and on parameters the~$\beta$,  $\rho$, and $\alpha$ we will introduce later (c.f.~\Cref{sec:const-and-params} and~\Cref{thm:optimal-spread-theorem}).  To avoid cumbersome notation, when bounding expressions of the form $D_x^a f$ we will also not track how the implicit constants depend on $a$.  (In practice, we only need a finite number of derivatives in all of our arguments, so we can always choose an implicit constant independent of $a$.)   When the implicit constant depends on additional quantities, we record them as a subscript: e.g.~$A \lesssim_\phi B$ indicates that the constant may depend on $\phi$.

\subsubsection{Function spaces}
For $1 \leq p \leq \infty$ and $s \in \bbR$, we let $L^p = L^p(\bbR^3)$ and $H^s = H^s(\bbR^3)$ denote the usual Lebesgue and (inhomogeneous) Sobolev spaces, with norms
\begin{equation*}
    \lVert f \rVert_{L^p}^p = \int_{\bbR^3} |f(x)|^p\;dx, \qquad \lVert f \rVert_{H^s}^2 = \int_{\bbR^3} \jBra{\xi}^{2s} |\hat{f}(\xi)|^2\;d\xi.
\end{equation*}
For mixed-norm spaces we write $L^q_t L^p_x$, $L^\infty_t H^s_x$, etc., with the convention that the time variable is taken first.  The Schwartz space is denoted $\cS = \cS(\bbR^3)$.

\subsubsection{Fourier transform}  We use the convention
\begin{equation*}
    \cF f(\xi) = \hat{f}(\xi) = \frac{1}{(2\pi)^{3/2}} \int_{\bbR^3} e^{-i x \cdot \xi} f(x)\;dx, \qquad \cF^{-1} g(x) = \frac{1}{(2\pi)^{3/2}} \int_{\bbR^3} e^{i x \cdot \xi} g(\xi)\;d\xi,
\end{equation*}
so that $\cF$ is unitary on $L^2$.

\subsubsection{Operators}  
We denote the momentum operator by $p = -i \nabla_x$.  Functions of $p$ are defined as Fourier multipliers: 
$$m(p) f = \cF^{-1}(m(\xi) \hat{f}(\xi))$$
In particular, the Dollard propagator $U_D(t,s)$ defined in~\eqref{eqn:Dollard-dyn-def} is a Fourier multiplier with the explicit symbol given in~\eqref{eqn:Dollard-Fourier-rep}.  The full propagator associated with~\eqref{eqn:main} is denoted $U(t,s)$, so that $u(t) = U(t,0) u_0$ solves~\eqref{eqn:main} with $u(0) = u_0$.  For an operator $A$, we write $A^*$ for its adjoint and $\ad_A(B) = [A,B] = AB - BA$ for the commutator.

%\subsubsection{Limits in operator topology}  Following Reed and Simon (Functional Analysis, Section VI.1), we denote by $\slim$ and $\wlim$ the limits in the strong and weak operator topology, respectively.  Concretely, for a family $\{A(t)\}_{t \in \bbR}$ of bounded operators on $L^2$,
%\begin{equation*}
%    A = \slim_{t \to \infty} A(t) \iff \lim_{t \to \infty} \lVert A(t) f - A f \rVert_{L^2} = 0 \text{ for every } f \in L^2,
%\end{equation*}
%\begin{equation*}
%    A = \wlim_{t \to \infty} A(t) \iff \lim_{t \to \infty} \langle g, A(t) f \rangle_{L^2} = \langle g, A f \rangle_{L^2} \text{ for all } f, g \in L^2.
%\end{equation*}

\subsubsection{Cutoff functions}  
We fix a smooth cutoff $F \in C^\infty(\bbR;[0,1])$ such that
\begin{equation*}
    F(s \leq 1) = \begin{cases} 1, & s \leq 1, \\ 0, & s \geq 2, \end{cases} \qquad F(s \geq 1) = 1 - F(s \leq 1),
\end{equation*}
We will always assume that $F$ is defined in such a way that $\sqrt{F(s\geq 1)}$ and $\sqrt{F'(s \geq 1)}$ are smooth functions (we refer the reader to~\cite[Section 2.1]{sofferScatteringLocalizedStates2024a} for an explicit construction having these properties).

%\subsubsection{Decomposition of the potential}  Given $\beta \in (0,1)$ to be chosen later, the exterior and interior parts of $V$ are
%\begin{equation*}
%    \Vext(x,t) = F\!\left(\frac{|x|}{t^\beta} \geq 1\right) V(x,t), \qquad \Vint(x,t) = V(x,t) - \Vext(x,t),
%\end{equation*}
%and the corresponding integrated phase correction is $\nu(\xi,t) = \int_0^t V(2 \xi \tau, \tau)\;d\tau$.

\subsubsection{Constants and parameters} \label{sec:const-and-params}

Throughout the paper, we will use the constants and parameters below:
\begin{itemize}
    \item The constant $\mu$ is the decay rate of the potential $V$.
    \item The parameter $\beta$, which we introduce in~\Cref{sec:Dollard-WO}, is used to separate the potential $V$ into an interior potential $\Vint$ and an exterior potential $\Vext$.  
    \item The parameter $\rho$ controls the rate of growth of $|x - 2pt|$.
\end{itemize}
For ease of reference, we have given the bounds we will assume on these parameters in~\Cref{tab:constants}.  For each $\mu \in \left(\frac{1}{2}, 1\right]$, it is possible to find $\beta$ and $\rho$ satisfying the bounds in~\Cref{tab:constants}; however, as $\mu \to \frac{1}{2}$, $\beta \to 1$ and $\rho \to \frac{1}{2}$, which reflects the fact that the Dollard modifier is insufficient for $\mu \leq \frac{1}{2}$.
\begin{table}[b]
    \begin{tabular}{|c|c|c|}
    \hline
    \textbf{Constant} & \textbf{Value or range} & \textbf{Introduced in. . .}\\\hline\hline
    $\mu$ & $\frac{1}{2} < \mu \leq 1$ & Equation~\eqref{eqn:V-cond}\\\hline
    $\alpha$ & $\alpha > \beta$ ($\alpha \geq \frac{2}{2+\mu}$ in~\Cref{thm:optimal-spread-theorem}) & \Cref{prop:rad-mom-vanishes} (\Cref{thm:optimal-spread-theorem} with $\alpha \geq \frac{2}{2+\mu}$) \\\hline
    $\beta$ & $\frac{3}{2+2\mu} < \beta < 1$ & Equations~\cref{eqn:V-ext-def,eqn:V-int-def} in~\Cref{sec:Dollard-WO}\\\hline
    $\rho$ & $2 - (1+\mu)\beta < \rho < \frac{1}{2}$ & \Cref{prop:ESG} \\\hline
    \end{tabular}
    \caption{\label{tab:constants} The constants and parameters used in the paper.}
\end{table}

\subsection{Preliminary results}

We record here a few standard facts that will be used repeatedly.  %  The first is a consequence of the assumption~\eqref{eqn:V-cond}.
%
%\begin{lemma}[Boundedness of $V(\cdot,t)$ on Sobolev spaces]\label[lemma]{lem:V-Sobolev-bd}
%    Under the hypothesis~\eqref{eqn:V-cond}, multiplication by $V(\cdot,t)$ is bounded on $L^2$ and on $H^1$ uniformly in $t$:
%    \begin{equation*}
%        \sup_{t \in \bbR} \left( \lVert V(\cdot,t) f \rVert_{L^2} + \lVert V(\cdot, t) f \rVert_{H^1} \right) \lesssim \lVert f \rVert_{H^1}.
%    \end{equation*}
%\end{lemma}
%\begin{proof}[Sketch]
%    The bound on $L^2$ is immediate from the $L^\infty$ bound $\sup_{x,t} |V(x,t)| \lesssim 1$.  For the $H^1$ bound, we use $|\nabla_x V| \lesssim \jBra{x}^{-1-\mu} \lesssim 1$ together with the product rule.
%\end{proof}
We begin with a result on the propagator $U(t,s)$ for~\eqref{eqn:main}:

\begin{lemma}[Conservation of $L^2$ and propagation of $H^1$]\label[lemma]{lem:L2-H1-bd}
    The propagator $U(t,s)$ is unitary on $L^2$.  Moreover, if $u_0 \in H^1$, then $u(t) = U(t,0) u_0 \in H^1$ for all $t \in \bbR$ and
    \begin{equation*}
        \lVert u(t) \rVert_{H^1} \lesssim e^{C |t|} \lVert u_0 \rVert_{H^1}
    \end{equation*}
    for some constant $C > 0$ depending only on $V$.
\end{lemma}
\begin{proof}
    Conservation of $L^2$ follows from the self-adjointness of $-\Delta + V(\cdot,t)$ for each fixed $t$.  The $H^1$ bound follows from differentiating $\lVert \nabla u \rVert_{L^2}^2$ in time and using the fact that $V \in L^\infty_t W^{1,\infty}_x$ together with Gronwall's inequality.
\end{proof}

For later use, we also record an abstract symmetrization formula for products of operators:
\begin{lemma}\label[lemma]{lem:three-term-sym}
    Suppose $[A,C] = 0$.  Then,
    \begin{equation*}
        A^2BC^2 + C^2 B A^2 = 2(AC)B(AC) + R(A,B,C)
    \end{equation*}
    where
    \begin{equation*}R(A,B,C) = A[[A,B],C]C + C[[C,B],A]{A} + A[C,[C,B]]A + C[A,[A,B]]C\end{equation*}
    involves double commutators of $A$ and $C$ with $B$.
\end{lemma}

\begin{remark}
    In practice, we will apply~\Cref{lem:three-term-sym} with $B = \partial_x^2$ and multiplication operators $A = f(\lambda^{-1} X), C = g(\lambda^{-1} X)$ for $f,g \in W^{2,\infty}$.  In this case, the remainder term $R(A,B,C)$ satisfies
    \begin{equation*}
        \lVert R(A,B,C) \rVert_{L^2 \to L^2} \lesssim \lambda^{-2}
    \end{equation*}
\end{remark}

\section{Existence of the Dollard wave operator} \label{sec:Dollard-WO}

In this section, we will prove the existence of the (weak) adjoint wave operator
\begin{equation}\label{eqn:full-adj-WO-def}
    \Omega^*_D = \wlim_{t \to \infty} U_D(0,t) U(t,0)
\end{equation}
\begin{proposition}\label[proposition]{prop:weak-WO}
    The limit in~\eqref{eqn:full-adj-WO-def} defining $\Omega^*_D$ exists in $L^2$.  Moreover, if $u(t) = U(t,0) u_0$ satisfies the energy bound~\eqref{eqn:energy-bound}, then $\Omega^*_D u_0 \in H^1$.
\end{proposition}
In fact, we will obtain this result as a consequence of the following corollary:
\begin{corollary}\label[corollary]{cor:full-WO}
    The wave operator
    \begin{equation}\label{eqn:full-WO-def}
        \Omega_D = \slim_{t \to \infty} U(0,t) U_D(t,0)
    \end{equation}
    exists in $L^2$.
\end{corollary}

Assuming for the moment that~\Cref{cor:full-WO} holds, the proof of~\Cref{prop:weak-WO} follows from functional analytic arguments:
\begin{proof}[Proof of~\Cref{prop:weak-WO} (assuming~\Cref{cor:full-WO})]
    Suppose $u_0 \in L^2$.  To prove the existence of $\Omega^*_D$, it suffices to prove that the weak limit of $U_D(0,t) U(t,0) u_0$ exists as $t$ goes to infinity. Now, since both $U$ and $U_D$ are unitary, $U_D(0,t) U(t,0) u_0$ is bounded in $L^2$, so by Banach-Alaoglu it suffices to prove that all weak subsequential limits agree.  In particular, the result will follow if we can show that
    \begin{equation*}
        \lim_{t \to \infty} \left\langle \phi, U_D(0,t) U(t,0) u_0\right\rangle_{L^2}
    \end{equation*}
    exists for all $\phi \in L^2$.  But~\Cref{cor:full-WO} implies that 
    \begin{equation*}
        \lim_{t \to \infty} \left\langle \phi, U_D(0,t) U(t,0) u_0\right\rangle_{L^2} = \langle \Omega_D \phi, u_0 \rangle_{L^2}
    \end{equation*}
    To prove that $\Omega^*_D u_0 \in H^1$ if $U(t,0)u_0$ satisfies the energy bound~\eqref{eqn:energy-bound}, note that for any $\phi \in H^1$,
    \begin{equation*}\begin{split}
        \langle \phi, \Omega^*_D(t) u_0 \rangle_{H^1} =& \langle P_{\leq j} U_D(t,0) \phi, U(t,0) u_0 \rangle_{H^1} + \langle P_{> j} U_D(t,0) \phi, U(t,0) u_0 \rangle_{H^1}\\
        =& \langle  U_D(t,0) \jBra{-\Delta} P_{\leq j} \phi, U(t,0) u_0 \rangle_{L^2} + o_{j \to \infty} \left( \lVert \phi \rVert_{H^1} \lVert U(t,0) u_0 \rVert_{L^\infty_t H^1_x} \right)\\
        =& \langle \Omega_D(t) \jBra{-\Delta} P_{\leq j} \phi, u_0 \rangle_{L^2} + o_{j \to \infty} \left( \lVert \phi \rVert_{H^1} \lVert U(t,0) u_0 \rVert_{L^\infty_t H^1_x} \right)
    \end{split}\end{equation*}
    uniformly in $t$, where
    \begin{equation*}
        \Omega^*_D(t) = U_D(0,t) U(t,0),\qquad \Omega_D(t) = U(0,t) U_D(t,0)
    \end{equation*}
    Thus, for any two times $t_1$ and $t_2$, we see that
    \begin{equation*}\begin{split}
        \left|\langle \phi, \left(\Omega^*_D(t_2) - \Omega^*_D(t_1)\right) u_0 \rangle_{H^1}\right| =&  \left\langle \left(\Omega_D(t_2)  - \Omega_D(t_1)\right)\jBra{-\Delta} P_{\leq j} \phi, u_0 \right\rangle_{L^2}\\
        &+ o_{j \to \infty} \left( \lVert \phi \rVert_{H^1} \rVert_{L^2} \lVert U(t,0) u_0 \rVert_{L^\infty_t H^1_x} \right)
    \end{split}\end{equation*}
    %{\color{red}XW: Maybe we can directly use the relation $\|U_D(0,t)U(t,0)u_0\|_{H^1}=\|U(t,0)u_0\|_{H^1}$ to conclude $\Omega_D^*u_0\in H^1$ since $U_D(0,t)$ is a pure differential operator.  GS: I don't think this is enough -- that gives that $\Omega^*_D(t) u_0$ is uniformly bounded in $H^1$, but not the required convergence.} 
    In particular, by first choosing $j$ large, and then choosing $t_1$, $t_2$ large enough (depending only on $j$ and $\phi$), we can make this term as small as we please, so $\langle \phi, \Omega^*_D(t) u_0 \rangle_{H^1}$ is Cauchy, and $\Omega^*_D u_0 = \wlim_{t \to \infty} \Omega^*_D(t) u_0$ is in $H^1$.
\end{proof}
Now, we will prove~\Cref{cor:full-WO}.  Since the propagators $U(t,s)$ and $U_D(t,s)$ are unitary, it suffices to prove that
\begin{equation}\label{eqn:WO-limit-2}
    \slim_{t \to \infty} U(0,t) U_D(t,0) \phi
\end{equation}
exists in $L^2$ for a dense set of $\phi$.  In particular, we can assume that $\phi$ is Schwartz class and that
\begin{equation}\label{eqn:phi-supp-hypo}
    \supp \hat{\phi} \subset \bbR^3 \setminus B_\epsilon(0)
\end{equation}
for some $\epsilon > 0$.

To prove that~\eqref{eqn:WO-limit-2} exists, we will use Cook's method.  Evidently,
\begin{equation*}\begin{split}
    U(0,t) U_D(t,0) \phi    =& \phi + \int_0^t \frac{d}{ds} \left(U(0,s) U_D(s,0) \phi\right) \;ds\\
                            =& \phi + i\int_0^t U(0,s) (V(x,s) - V(2ps, s)) U_D(s,0) \phi\;ds
\end{split}\end{equation*}
so it suffices to prove that $\lVert (V(x,t) - V(2pt,t)) U_D(t,0) \phi \rVert_{L^2_x}$ is integrable in $t$.  Since $V$ is bounded, we immediately have the bound
\begin{equation*}
    \lVert (V(x,t) - V(2pt,t)) U_D(t,0) \phi \rVert_{L^2_x} \lesssim \lVert \phi \rVert_{L^2_x}
\end{equation*}
so we will be finished once we show that this expression also decays sufficiently rapidly as $t \to \infty$.

Here and later in the paper, it will be helpful to divide the potential $V(x,t)$ into an exterior potential
\begin{equation}\label{eqn:V-ext-def}
    \Vext(x,t) = F\left(\frac{|x|}{t^\beta} \geq 1\right) V(x,t)
\end{equation}
and an interior potential
\begin{equation}\label{eqn:V-int-def}
    \Vint(x,t) = V(x,t) - \Vext(x,t)
\end{equation}
Throughout the paper, we will assume that $\beta \in \left(\frac{3}{2+2\mu}, 1\right)$ (c.f.~\Cref{tab:constants}).

Based on the support assumptions on $\phi$, for $t$ sufficiently large we have that $\Vint(2pt,t) \phi = 0$, so we can rewrite~\eqref{eqn:Dollard-Fourier-rep} as
\begin{equation*}\begin{split}
    (V(x,t) - V(2pt,t)) U_D(t,0) \phi =  \Vint(x,t) U_D(t,0) \phi + (\Vext(x,t) - \Vext(2pt,t))U_D(t,0) \phi
\end{split}\end{equation*}
We now show that each term on the right decays in $L^2$:

\subsection{Interior bounds}

We first obtain bounds for $\Vint(x,t) U_D(t,0)\phi$.  Using~\eqref{eqn:Dollard-Fourier-rep}, we see that this term can be rewritten as
\begin{equation}\label{eqn:int-bound-osc-int}
    \Vint(x,t) U_D(t,0)\phi = \frac{\Vint(x,t)}{(2\pi)^{3/2}} \int_{\bbR^3}e^{i\Phi}\hat{\phi}(\xi)\;d\xi
\end{equation}
where the phase $\Phi$ is given by
\begin{equation*}
    \Phi = \Phi(\xi;x,t) = x \cdot \xi-t |\xi|^2 - \int_0^t V(2\xi \tau, \tau)\;d\tau
\end{equation*}
Differentiating in $\xi$, we see that
\begin{equation*}
    \nabla_\xi \Phi = x - 2t\xi - \nabla_\xi \nu(\xi,t)
\end{equation*}
where
\begin{equation*}
    \nu(\xi,t) = \int_0^t V(2\xi \tau, \tau)\;d\tau
\end{equation*}
The term $\nu$ provides the Dollard correction in the phase.  For our proof, it will be important that $\nu$ is lower-order from the perspective of stationary phase estimates.  More precisely, based on the decay hypothesis on $V$ and $\nabla V$, we have that for $t \geq 1$, $\xi \in \supp \phi$,
\begin{equation*}
    |\nabla_\xi \nu(\xi,t)| = \left| \int_0^t 2\tau \nabla_x V(2\xi \tau, \tau)\;d\tau\right| \lesssim_\phi |\xi|^{-1-\mu} \begin{cases}
        t^{1-\mu} & \mu < 1\\
        \log(2+t) & \mu =1
    \end{cases}
\end{equation*}
In particular, for $t$ sufficiently large, 
\begin{equation}\label{eqn:nu-deriv-bound}
    \sup_{\xi \in \supp \phi} |\nabla_\xi \nu(\xi,t)| < \frac{1}{2} t^{1/2}
\end{equation}
In particular, for $x \in \supp \Vint(x,t)$, $\xi \in \supp \hat{\phi}$, we see that
\begin{equation*}
    |\nabla_\xi \Phi| \gtrsim t|\xi|
\end{equation*}
Defining $\mathcal{L} = -i\frac{\nabla_\xi \Phi}{|\nabla_\xi \Phi|^2} \cdot \nabla_\xi$, we see that $\mathcal{L} e^{i\Phi} =e^{i\Phi}$, so repeated integration by parts yields the estimate
\begin{equation*}\begin{split}
    \left| \Vint(x,t) U_D(t,0) \phi(x) \right| \lesssim& |\Vint(x,t)| \int_{\bbR^3} |(\mathcal{L}^*)^n \hat{\phi}(\xi)|\;d\xi\\
                                            \lesssim_{n,\phi}& |\Vint(x,t)| t^{-n}
\end{split}\end{equation*}
From here, a simple volume bound gives that
\begin{equation*}
    \lVert \Vint(x,t) U_D(t,0) \phi \rVert_{L^2} \lesssim_{n,\phi} t^{\left(\frac{3}{2}- \mu\right)\beta - n}
\end{equation*}
so this term decays rapidly.
\subsection{Exterior bounds}\label{sec:exterior-bounds}

\subsubsection{A preliminary decomposition}

As before, we rewrite the quantity we are considering as an oscillatory integral: 
\begin{equation*}
    \left(\Vext(x,t) - \Vext(2pt,t)\right) U_D(t,0) \phi = \frac{1}{(2\pi)^{3/2}} \int_{\bbR^3}e^{i\Phi}\left(\Vext(x,t) - \Vext(2\xi t,t)\right) \hat{\phi}(\xi)\;d\xi
\end{equation*}
By the Fundamental Theorem of Calculus, we have that
\begin{equation}\label{eqn:Dollard-cancellation}\begin{split}
\Vext(x,t) - \Vext(2\xi t,t) =& (x - 2 \xi t) \cdot \int_0^1 \nabla_x \Vext\left(\lambda x + 2(1-\lambda)\xi t\right)\; d\lambda\\
=& (\nabla_\xi \Phi + \nabla_\xi \nu) \cdot \int_0^1 \nabla_x \Vext\left(\lambda x + 2(1-\lambda)\xi t\right)\; d\lambda\\
=:& (\nabla_\xi \Phi + \nabla_\xi \nu) \cdot \CVext(x,\xi,t)
\end{split}\end{equation}
which allows us to re-express the previous integral as
\begin{equation}\label{eqn:V-ext-expr-1}
    \left(\Vext(x,t) - \Vext(2pt,t)\right) U_D(t,0) \phi = \frac{1}{(2\pi)^{3/2}} \int_{\bbR^3}e^{i\Phi}\left(\nabla_\xi \Phi + \nabla_\xi \nu\right) \cdot \CVext(x,\xi,t) \hat{\phi}(\xi)\;d\xi
\end{equation}
In this integral, we can think of $\nabla_\xi \nu$ as a lower-order error term.  More precisely, we have the following theorem:
\begin{lemma}\label[lemma]{lem:nu-deriv-slow-growth}
    For each $\epsilon > 0$, there exist constants $C_\epsilon$ and $T_\epsilon$ such that for all $t \geq T_\epsilon$, we have the bound
    \begin{equation*}
        \left|D_\xi^a \nu(\xi,t) \right| \leq C_\epsilon t^{1-\mu}
    \end{equation*}
    uniformly for $|\xi| > \epsilon$ and $a = 1,2,3$.
\end{lemma}
\begin{proof}
    We will prove the result for $\nabla_\xi \nu$: The corresponding result for higher derivatives is similar.  By direct calculation, we have that
    \begin{equation*}\begin{split}
        |\nabla_\xi \nu(\xi,t)| =& \left| \int_0^t 2\tau \nabla_x V(2\tau\xi, \tau)\;d\tau\right|\\
                                \lesssim& \int_0^t 2\tau \jBra{2\tau \xi}^{-1-\mu} \;d\tau
    \end{split}\end{equation*}
    Now, for $t > 100 \epsilon^{-1} = T_\epsilon$ and $|\xi| > \epsilon$, we see that the integrand is $O(\tau)$ over the interval $[0,T_\epsilon]$ and $O(\tau^{-\mu} |\xi|^{-1-\mu})$ for $\tau > T_\epsilon$.  Thus, 
    \begin{equation*}\begin{split}
        |\nabla_\xi \nu(\xi,t)| \lesssim& T_\epsilon^2 + |\xi|^{-1-\mu} t^{1-\mu}\\
                                \lesssim& \left(T_\epsilon^{\mu} + \epsilon^{-1-\mu}\right) t^{1-\mu}\\
                                \leq& C_\epsilon t^{1-\mu}
    \end{split}\end{equation*}
    as required.  Examining the argument, we see that higher derivatives can be handled similarly by redefining the constant $C_\epsilon$.
\end{proof}
Since $\mu > 1/2$, we have that $|\nabla_\xi \nu| \ll_\epsilon t^{1/2}$ for $t$ sufficiently large.  Thus, in the region where
\begin{equation*}
    |\nabla_\xi \Phi| \gtrsim t^{1/2}
\end{equation*}
we have that
\begin{equation*}
    \nabla_\xi \Phi = \nabla_\xi \Phi_0 + O_\epsilon(t^{1-\mu})
\end{equation*}
where $\Phi_0 = \Phi + \nu$ is the phase for the free Schr\"odinger equation:
\begin{equation}\label{eqn:Phi-0-def}
    \Phi_0(\xi;x,t) = x\cdot \xi - t|\xi|^2
\end{equation}
In particular, if we are far from the stationary points in the sense that
\begin{equation*}
    |\nabla_\xi \Phi| \gtrsim t^{1/2}
\end{equation*}
then the Dollard correction is lower order.  Thus, we will rewrite~\eqref{eqn:V-ext-expr-1} by separating out the lower order terms involving $\nabla \nu$ from the leading order expression containing $\nabla \Phi$, and then integrate by parts in the leading order terms using the identity
\begin{equation*}
    \nabla_\xi e^{i\Phi} = i \nabla_\xi \Phi e^{i\Phi}
\end{equation*}
to obtain:
\begin{equation*}\begin{split}
    \left(\Vext(x,t) - \Vext(2pt,t)\right) U_D(t,0) \phi =& -\frac{i}{(2\pi)^{3/2}} \int_{\bbR^3}e^{i\Phi} \nabla_\xi \cdot \left(\CVext(x,\xi,t) \hat{\phi}(\xi)\right)\;d\xi\\
    &+ \frac{1}{(2\pi)^{3/2}} \int_{\bbR^3}e^{i\Phi}\nabla_\xi \nu \cdot \CVext(x,\xi,t) \hat{\phi}(\xi)\;d\xi\\
    =:& \rmI + \rmII
\end{split}\end{equation*}
We then further divide these terms into their near stationary components (where $|\nabla \Phi_0(\xi)| \lesssim t^{1/2}$):
\begin{equation}\label{eqn:near-stat}\begin{split}
    \rmI_\text{stat}(x,t) = \frac{i}{(2\pi)^{3/2}} \int_{\bbR^3}e^{i\Phi} F(t^{1/2}|\xi - \xi_\text{stat}| \leq 1) \nabla_\xi \cdot \left(\CVext(x,\xi,t) \hat{\phi}(\xi)\right)\;d\xi\\
    \rmII_\text{stat}(x,t) = \frac{1}{(2\pi)^{3/2}} \int_{\bbR^3}e^{i\Phi}F(t^{1/2}|\xi - \xi_\text{stat}| \leq 1) \nabla_\xi \nu \cdot \CVext(x,\xi,t) \hat{\phi}(\xi)\;d\xi
\end{split}\end{equation}
and their nonstationary components:
\begin{equation}\label{eqn:non-stat}\begin{split}
    \rmI_\text{nonstat}(x,t) = \frac{i}{(2\pi)^{3/2}} \int_{\bbR^3}e^{i\Phi} F(t^{1/2}|\xi - \xi_\text{stat}| \geq 1) \nabla_\xi \cdot \left(\CVext(x,\xi,t) \hat{\phi}(\xi)\right)\;d\xi\\
    \rmII_\text{nonstat}(x,t) = \frac{1}{(2\pi)^{3/2}} \int_{\bbR^3}e^{i\Phi}F(t^{1/2}|\xi - \xi_\text{stat}| \geq 1) \nabla_\xi \nu \cdot \CVext(x,\xi,t) \hat{\phi}(\xi)\;d\xi
\end{split}\end{equation}
Here, $\xi_\text{stat}(x,t) = \frac{x}{2t}$ is the stationary point for the \textit{free} phase $\Phi_0$.  We will obtain decay in time for $\rmI_\text{stat}, \rmII_\text{stat}$ in $L^2_x$ in~\Cref{sec:near-stat-decay} and then prove time decay for the terms $\rmI_\text{nonstat}$ and $\rmII_\text{nonstat}$ in~\Cref{sec:non-stat-decay}.

\subsubsection{Near stationary points}\label{sec:near-stat-decay}

Near the stationary point $\xi_\text{stat}$, the strategy here is to exploit the fact that $2\xi t \approx x$ in order to prove decay for $\CVext$ as $x \to \infty$:
\begin{lemma}\label[lemma]{lem:CVext-decay-near-stat}
    Let $|\xi| > \epsilon$.  If $|\xi - \xi_\text{stat}| \lesssim t^{-1/2}$, then for $t$ sufficiently large (depending on $\epsilon$), the bounds
    \begin{equation}\label{eqn:CVext-decay-near-stationary}
        |\CVext(x,\xi,t)| \lesssim_\epsilon \jBra{x}^{-1-\mu}
    \end{equation}
    and
    \begin{equation}\label{eqn:CVext-deriv-decay-near-stationary}
        |D_\xi^a\CVext(x,\xi,t)| \lesssim_\epsilon \jBra{x}^{-1-\mu},\qquad 1 \leq a \leq 3
    \end{equation}
    hold uniformly in $\xi$ and $t$.
\end{lemma}
\begin{proof}
    We will first prove~\eqref{eqn:CVext-decay-near-stationary}, then explain how to accommodate the derivatives in~\eqref{eqn:CVext-deriv-decay-near-stationary}.  Observe that
    \begin{equation*}
        \lambda x + 2(1-\lambda)\xi t = (\lambda - 1) (x - 2\xi t) + x
    \end{equation*}
    Since $x - 2\xi t = 2t(\xi_\text{stat} - \xi)$, the first term is $O(t^{1/2})$.  On the other hand, since 
    $$x = 2t\xi_\text{stat} = 2t\xi + O(t^{1/2})$$
    for all sufficiently large times $t$ we have that
    \begin{equation*}
        |x| \geq \epsilon t \gg t^{1/2}
    \end{equation*}
    In particular, for $t$ large enough
    \begin{equation*}
        |\lambda x + 2(1-\lambda)\xi t| \sim |x|
    \end{equation*}
    and thus
    \begin{equation*}
        |\CVext(x,\xi,t)| \leq \int_0^1 |\nabla_x V(\lambda x + 2(1-\lambda)\xi t)|\;d\lambda \lesssim \jBra{x}^{-1-\mu}
    \end{equation*}
    which completes the proof of~\eqref{eqn:CVext-decay-near-stationary}.  Turning to~\eqref{eqn:CVext-deriv-decay-near-stationary}, we observe that

    \begin{equation*}
        D_\xi^a \CVext(x,\xi,t) = \int_0^1 (2(1-\lambda) t)^{a} D^a \nabla_x V(\lambda x + 2(1-\lambda)\xi t)\;d\lambda 
    \end{equation*}
    so
    \begin{equation*}
        |D_\xi^a \CVext(x,\xi,t)| \lesssim |t|^{a} \jBra{x}^{-1-a -\mu}  \lesssim_\epsilon \jBra{x}^{-1-\mu}
    \end{equation*}
    where the last inequality follows from the fact that $|x| \geq \epsilon t$.
\end{proof}
Combining the first inequality in~\Cref{lem:CVext-decay-near-stat} with~\Cref{lem:nu-deriv-slow-growth}, we obtain the pointwise bound
\begin{equation*}\begin{split}
    |\rmII_\text{stat}(x,t)| \lesssim& \lVert F(t^{1/2}|\xi - \xi_\text{stat}| \leq 1) \rVert_{L^1_\xi} \lVert \hat{\phi} \nabla_\xi \nu \rVert_{L^\infty_\xi} \lVert \CVext(x,\xi,t) \rVert_{L^\infty_\xi}\\
        \lesssim_\phi& t^{-3/2} t^{1-\mu} \jBra{x}^{-1-\mu}\\
        \lesssim_\phi& t^{-1/2-\mu} \jBra{x}^{-1-\mu}
\end{split}\end{equation*}
for $t$ sufficiently large depending on $\epsilon$ (and hence on $\phi$).  In particular, since $\mu > 1/2$, we see that
\begin{equation*}
    \lVert\rmII_\text{stat}(x,t) \rVert_{L^2} \lesssim_\phi t^{-1/2-\mu}
\end{equation*}
is integrable in time over $[1,\infty)$.  By also using the derivative bound in~\Cref{lem:CVext-decay-near-stat}, we find that
\begin{equation*}\begin{split}
   |\rmI_\text{stat}(x,t)| \lesssim& \lVert F(t^{1/2}|\xi - \xi_\text{stat}| \leq 1) \rVert_{L^1_\xi} \lVert \nabla_\xi \cdot(\hat{\phi} \CVext(x,\xi,t)) \rVert_{L^\infty_\xi}\\
   \lesssim_\phi& t^{-3/2} \jBra{x}^{-1-\mu}
\end{split}\end{equation*}
so
\begin{equation*}
    \lVert \rmI_\text{stat}(x,t) \rVert_{L^2_x(\bbR^3)} \lesssim t^{-3/2}
\end{equation*}
is also integrable in time.

\subsubsection{Far from stationary points}\label{sec:non-stat-decay}
To handle the contribution coming from a distance greater than $t^{-1/2}$ from the stationary point, we will use the nonstationary phase principle.  Let us define the functions $\chi_j$ by
\begin{equation*}
    \chi_j(z) = F(2^{-j-1} |z| \leq 1) - F(2^{-j} |z| \leq 1)
\end{equation*}
In particular, $\chi_j(z)$ vanishes unless $|z| \in [2^j, 2^{j+2}]$.  Observe that
\begin{equation*}
    F(|z| \geq 1) = \sum_{j = 0} ^\infty \chi_j(z)
\end{equation*}
and
\begin{equation}\label{eqn:chi-j-supp}
    \supp \chi_j \subset [2^j, 2^{j+2}]
\end{equation}
so we can rewrite $\rmInst$ and $\rmIInst$ as
\begin{equation*}\begin{split}
    \rmInst =& \sum_{j=0}^\infty \frac{i}{(2\pi)^{3/2}} \int_{\bbR^3}e^{i\Phi} \chi_j(t^{1/2}|\xi - \xi_\text{stat}|) \nabla_\xi \cdot \left(\CVext(x,\xi,t) \hat{\phi}(\xi)\right)\;d\xi\\
            =:& \sum_{j=0}^\infty \rmInst^j\\
    \rmIInst =& \sum_{j=0}^\infty \frac{1}{(2\pi)^{3/2}} \int_{\bbR^3}e^{i\Phi}\chi_j(t^{1/2}|\xi - \xi_\text{stat}|) \nabla_\xi \nu \cdot \CVext(x,\xi,t) \hat{\phi}(\xi)\;d\xi\\
            =:& \sum_{j=0}^\infty \rmIInst^j
\end{split}\end{equation*}
Thus, it suffices to obtain bounds for $\rmInst^j$ and $\rmIInst^j$ that are summable in $j$.  When doing this, it is helpful to distinguish the cases $j < J$ and $j \geq J$, where $J$ is defined by
\begin{equation}\label{eqn:J-def}
    \frac{\epsilon t^{1/2}}{2^{10}} \leq 2^J < \frac{\epsilon t^{1/2}}{2^9}
\end{equation}
The parameter $J$ represents the threshold above which we no longer get useful cancellations for $\CVext$:
\begin{lemma}\label[lemma]{lem:CVext-nonstat-lemma}
    Suppose $\xi \in \supp \chi_j(t^{1/2} |\bullet - \xi_\text{stat}|) \cap \{|\xi| > \epsilon\}$.  If $j < J$ or $j \geq J$ and $|x| > 2^{j+10} t^{1/2}$, then for $t$ sufficiently large we have the bound
    \begin{equation}\label{eqn:CVext-decay-non-stat-below}
        |D_\xi^a\CVext(x,\xi,t)| \lesssim \jBra{x}^{-1-\mu},\qquad 1 \leq a \leq 3
    \end{equation}
    Otherwise, for all sufficiently large $t$, we have that
    \begin{equation}\label{eqn:CVext-decay-non-stat-above}
        |D_\xi^a\CVext(x,\xi,t)| \lesssim t^{-(1+\mu)\beta},\qquad  1\leq a \leq 3
    \end{equation}
\end{lemma}
\begin{proof}
    To prove the first part of the result involving~\eqref{eqn:CVext-decay-non-stat-below}, it suffices to show that the condition
    \begin{equation} \label{eqn:x-vs-drift-cond}
        |x| > 2|x - 2\xi t|
    \end{equation}
    holds whenever $j < J$ or $j \geq J$, $|x| > \epsilon 2^{j+10} t^{1/2}$, since in this case we can repeat the argument from the proof of~\Cref{lem:CVext-decay-near-stat}. The condition $\xi \in \supp \chi_j(t^{1/2}|\bullet - \xi_\text{stat}|)$ implies that
    \begin{equation*}
        |x - 2\xi t| = 2t|\xi - \xi_\text{stat}| \leq 2^{j+3} t^{1/2}
    \end{equation*}
    so if $|x| > 2^{j+10} t^{1/2}$, we immediately see that~\eqref{eqn:x-vs-drift-cond} holds.  On the other hand, if $j < J$, then since $|\xi| > \epsilon$, we have that
    \begin{equation*}
        |x - 2\xi t| \leq 2^{j+3} t^{1/2} < 2^{-6}\epsilon t
    \end{equation*}
    so,
    \begin{equation*}
        |x| \geq 2t|\xi| - |x - 2\xi t| > \epsilon t > 2|x - 2\xi t|
    \end{equation*}
    In either case, we can run the argument from~\Cref{lem:CVext-decay-near-stat} to conclude that~\eqref{eqn:x-vs-drift-cond} is satisfied.  

    We now turn to~\eqref{eqn:CVext-decay-non-stat-above}.  Here, the bound is in fact unconditional: Based on the decay of $V$ and the definition of $\Vext$, we have that
    \begin{equation*}
        |D_x^a \Vext(x,t)| \lesssim t^{-(a + \mu)\beta }
    \end{equation*}
    so
    \begin{equation*}
        |D_\xi^a \CVext(x,\xi,t)| \lesssim_a \int_0^1 t^{a} |D_x^a \nabla_x \Vext(\lambda x + 2(1-\lambda)\xi t,t)|\;d\lambda\\
        \lesssim_a t^{-(1+\mu)\beta }\qedhere
    \end{equation*}
\end{proof}
With this bound in hand, we are ready to derive bounds for the $\rmInst^j$ and $\rmIInst^j$.  Define the differential operator
\begin{equation*}
    \mathcal{L} := -i \frac{\nabla_\xi \Phi}{|\nabla_\xi \Phi|^2} \cdot \nabla_\xi
\end{equation*}
and observe that
\begin{equation}\label{eqn:L-IBP-ident}
    \mathcal{L} e^{i\Phi} = e^{i\Phi}
\end{equation}
The adjoint is given by
\begin{equation}\label{eqn:L-adjoint}\begin{split}
    \mathcal{L}^* g =& -i \nabla_\xi \cdot  \left(\frac{\nabla_\xi \Phi}{|\nabla_\xi \Phi|^2} g\right)\\
    =& -i (\nabla_\xi \cdot F) g -i F \cdot \nabla_\xi g
\end{split}\end{equation}
where $F = \frac{\nabla_\xi \Phi}{|\nabla_\xi \Phi|^2}$.  Applying~\eqref{eqn:L-IBP-ident} twice and integrating by parts, we find that
\begin{equation*}
    \rmInst^j = \frac{i}{(2\pi)^{3/2}} \int_{\bbR^3}e^{i\Phi} (\mathcal{L}^*)^2\left(\chi_j\left(t^{1/2}|\xi - \xi_\text{stat}|\right) \nabla_\xi \cdot \left(\CVext(x,\xi,t) \hat{\phi}(\xi)\right)\right)\;d\xi
\end{equation*}
Now, on the support of $\chi_j\left(t^{1/2}|\bullet - \xi_\text{stat}|\right)$, Lemma~\ref{lem:nu-deriv-slow-growth} implies that for all sufficiently large $t$,
\begin{equation*}\begin{split}
    |\nabla_\xi \Phi| =& |\nabla_\xi \Phi_0 - \nabla_\xi \nu| \sim 2^j t^{1/2}\\
    |D^2_\xi \Phi| =& |D^2_\xi \Phi_0 - D^2_\xi \nu| \sim t\\
    |D^3_\xi \Phi| =& |D^3_\xi \nu| = O(t^{1-\mu})
\end{split}\end{equation*}
Combining these estimates with~\eqref{eqn:L-adjoint}, a straightforward calculation shows that
\begin{equation*}\begin{split}
    (\mathcal{L}^*)^2 g =& -(\nabla_\xi \cdot F)^2 g - (\nabla_\xi \cdot F) F\cdot \nabla_\xi g - F\cdot \nabla_\xi (\nabla_\xi \cdot F g)\\
                &- F \cdot \nabla_\xi F (F \cdot \nabla_\xi g)\\
                =& O(2^{-4j} g) + O(t^{-1/2-\mu} 2^{-3j} g) + O(t^{-1/2} 2^{-3j} |\nabla_\xi g|) + O(t^{-1} 2^{-2j} | D^2_\xi g|)
\end{split}\end{equation*}
on the support of $\chi_j(t^{1/2}(|\bullet - \xi_\text{stat}|)$.  Here, we are interested in the case 
$$g = \chi_j(t^{1/2} |\xi - \xi_\text{stat}|) \nabla_\xi \cdot (\CVext \hat{\phi})$$
To estimate this term, we note that we have the bounds
\begin{equation}\label{eqn:nabla-CVext-hphi-bd}
    \lVert D_\xi^a (\CVext \hat{\phi}) \rVert_{L^1_\xi \cap L^\infty_\xi} \lesssim_\phi \begin{cases}
        \jBra{x}^{-1-\mu} & j < J \text{ or } j \geq J \text{ and } |x| > 2^{j+10} t^{1/2}\\
        t^{-(1+\mu)\beta} & \text{else}
    \end{cases}
\end{equation}
and
\begin{equation}\label{eqn:cut-off-deriv-Lp}
    \lVert D_\xi^a \chi_j(t^{1/2}(|\bullet - \xi_\text{stat}|) \rVert_{L^p_\xi} \lesssim \left(t^{1/2} 2^{-j}\right)^{a} \left( t^{-1/2} 2^j\right)^{3/p}
\end{equation}
Taking $p = 1$ in~\eqref{eqn:cut-off-deriv-Lp} gives
\begin{equation*}
    \lVert 2^{-4j} g \rVert_{L^1_\xi} \lesssim_\phi t^{-3/2} 2^{-j}
\end{equation*}
Combining this with the estimates for $p = 3/2$ and $p = 3$ also gives
\begin{equation*}\begin{split}
    \lVert t^{-1/2} 2^{-3j} \nabla_\xi g\rVert_{L^1_\xi} + \lVert t^{-1} 2^{-2j} D^2_\xi g \rVert_{L^1_\xi} \lesssim_\phi& t^{-3/2} 2^{-j}\\
    \lVert t^{-1/2-\mu} 2^{-3j} g \rVert_{L^1_\xi} \lesssim& t^{-3/2-\mu} 2^{-j} \ll t^{-3/2} 2^{-j}
\end{split}\end{equation*}
Combining these bounds with~\eqref{eqn:nabla-CVext-hphi-bd}, we find that
\begin{equation}\label{eqn:rmInst-j-bd-1}\begin{split}
    |\rmInst^j(x,t)| \lesssim& \lVert (\mathcal{L}^*)^2 g \rVert_{L^1_\xi}\\
    \lesssim& t^{-3/2} 2^{-j} \begin{cases}
        \jBra{x}^{-1-\mu} & j < J \text{ or } j \geq J \text{ and } |x| > 2^{j+10} t^{1/2}\\
        t^{-1-\mu} & \text{else}
    \end{cases}
\end{split}\end{equation}
In particular, for $j < J$, we have that
\begin{equation*}
    \lVert \rmInst^j(x,t) \rVert_{L^2_x(\bbR^3)} \lesssim t^{-3/2} 2^{-j}
\end{equation*}
which is summable in $j$.  If $j \geq J$, then in the region $|x| > 2^{j+10} t^{1/2}$, the same reasoning as above shows that
\begin{equation*}
   \lVert \rmInst^j(x,t) \rVert_{L^2_x(|x| \geq 2^{j+10} t^{1/2})} \lesssim t^{-3/2} 2^{-j}
\end{equation*}
To handle the interior region $|x| < 2^{j+10} t^{1/2}$, we must modify the calculation above slightly.  Taking $p = 6/5+$ in~\eqref{eqn:cut-off-deriv-Lp} gives
\begin{equation*}
    \lVert 2^{-4j} g \rVert_{L^1_\xi} \lesssim_\phi t^{-5/4 +} 2^{(-3/2-) j }
\end{equation*}
and further taking $p = 2+, p = 6+$ yields
\begin{equation*}\begin{split}
    \lVert t^{-1/2} 2^{-3j} \nabla_\xi g\rVert_{L^1_\xi} + \lVert t^{-1} 2^{-2j} D^2_\xi g \rVert_{L^1_\xi} \lesssim_\phi& t^{-5/4+} 2^{(-3/2-)j}\\
    \lVert t^{-1/2-\mu} 2^{-3j} g \rVert_{L^1_\xi} \lesssim& t^{-5/4-\mu+} 2^{(-3/2-)j} \ll t^{-5/4+} 2^{(-3/2-)j}
\end{split}\end{equation*}
Applying~\eqref{eqn:nabla-CVext-hphi-bd} and arguing as above, we see that
\begin{equation}\label{eqn:interior-region-bound}\begin{split}
    \lVert \rmInst^j(x,t) \rVert_{L^2_x\left(|x| \leq 2^{j+10} t^{1/2}\right)} \lesssim& t^{3/4} 2^{3/2j} \lVert \rmInst^j(x,t) \rVert_{L^\infty_x\left(|x| \leq 2^{j+10} t^{1/2}\right)}\\
    \lesssim_\delta& t^{-\frac{1}{2} +\delta} t^{-(1+\mu)\beta} 2^{-\delta j}
\end{split}\end{equation}
for some $\delta > 0$.  Since $\beta > \frac{3}{2(1+\mu)}$, we see that
\begin{equation*}
    -\frac{1}{2} - (1+\mu)\beta < -1
\end{equation*}
Choosing $\delta$ sufficiently small depending on $\beta$ and summing over $j \geq 0$, we find that
\begin{equation*}
    \lVert \rmInst(x,t) \rVert_{L^2_x\left(|x| \leq 2^{j+10} t^{1/2}\right)} \in L^1_t(1,\infty)
\end{equation*}

Turning to the bounds for $\rmIInst^j$, we note that the integrand for $\rmIInst^j$ may be obtained from the integrand for $\rmInst^j$ by replacing the $\nabla_\xi (\CVext \hat{\phi})$ with $\nabla_\xi \nu \cdot (\CVext \hat{\phi})$.  In particular, by repeating the argument we used to bound $\rmInst^j$ but replacing~\eqref{eqn:nabla-CVext-hphi-bd} with the bound
\begin{equation}\label{eqn:nabla-nu-CVext-hphi-bd}
    |D_\xi^a(\nabla \nu \cdot \CVext \hat{\phi})| \lesssim_\alpha t^{1-\mu} \begin{cases}
        \jBra{x}^{-1-\mu} & j < J \text{ or } j \geq J \text{ and } |x| > 2^{j+10} t^{1/2}\\
        t^{-(1+\mu)\beta} & \text{else}
    \end{cases}
\end{equation}
(which follows from~\Cref{lem:nu-deriv-slow-growth,lem:CVext-nonstat-lemma}), we see that
\begin{equation}\label{eqn:rmIInst-bd}
    \lVert \rmIInst(x,t) \rVert_{L^2_x\left(|x| \leq 2^{j+10} t^{1/2}\right)} \lesssim t^{1/2 - \mu - (1+\mu)\beta +}
\end{equation}
which is also acceptable.  
\begin{remark}
    According to the bound above, $\rmIInst$ may decay more slowly than $\rmInst$, contradicting the heuristic that the Dollard phase terms are lower order.  The reason for the apparent discrepancy is that we have only integrated by parts twice in $\rmIInst$, while in $\rmInst$ we integrated three times.  By performing a third integration by parts in the expression for $\rmIInst^j$ and estimating the resulting terms carefully, it is possible to show that this term decays faster than $t^{-3/2}$, agreeing with the heuristic that it is subleading.  In the interest of exposition, we have elected to give only the weaker bound~\eqref{eqn:rmIInst-bd}.
\end{remark}

\subsection{Completing the argument}
Combining the results of the previous sections, we now conclude that
\begin{equation*}
    \lVert (V(x,t) - V(2pt,t)) U_D(t,0) \phi \rVert_{L^1(0,\infty; L^2(\bbR^3))} < \infty 
\end{equation*}
so Cook's method now gives the existence of the wave operator in~\Cref{cor:full-WO}.

\subsection{A decomposition using the wave operators}

Based on~\Cref{cor:full-WO}, we see that any solution $u(x,t)$ to~\eqref{eqn:main} decomposes as
\begin{equation*}
    u(x,t) = U_D(t,0) u_+(x) + \urem(x,t)
\end{equation*}
with $u_+ = \Omega^*_D u_0$.  In particular, this gives the decomposition~\eqref{eqn:main-decomp}.  Moreover, as required by~\eqref{eqn:remainder-weak-lim}, $\urem$ is asymptotically weakly orthogonal to the Dollard flow, since by the definition of $\Omega_D^*$
\begin{equation}\label{eqn:remainder-weak-limit-2}
    \wlim_{t \to \infty} U_D(0,t) \urem(x,t) = \wlim_{t \to \infty} U_D(0,t) u(x,t) - u_+ = 0
\end{equation}
If we also assume that the finite energy condition~\eqref{eqn:energy-bound} is satisfied, the second part of~\Cref{prop:weak-WO} guarantees that $u_+ \in H^1(\bbR^3)$, and since $U_D$ is an $H^1$ isometry, ${\urem(x,t) = u(x,t) - U_D(t,0) u_+}$ is in $L^\infty_t H^1_x$.  In particular, the decomposition ${u(x,t) = U_D(t,0) u_+(x) + \urem(x,t)}$ satisfies all the conditions required in~\Cref{thm:main-thm}.  

In order to prove the sublinear spreading estimate~\eqref{eqn:urem-slow-spreading} and the more accurate estimate in~\Cref{thm:better-rate}, it will be helpful to employ a slightly different decomposition of $u$ where each piece solves~\eqref{eqn:main}:
\begin{proposition}\label[proposition]{prop:u-redecomp}
    Suppose that $u(x,t) = U(t,0) u_0(x)$ is a solution to~\eqref{eqn:main} satisfying the finite energy condition~\eqref{eqn:energy-bound}.  Then, we can find two functions $\uscat, \uwl \in L^\infty_t H^1_x$ solving~\eqref{eqn:main} such that:
    \begin{enumerate}
        \item We have the decomposition
        \begin{equation}\label{eqn:u-redecomp}
            u(x,t) = \uscat(x,t) + \uwl(x,t)
        \end{equation}
        \item $\uscat$ is \emph{scattering} in the sense that
        \begin{equation}\label{eqn:uscat-scats}
            \slim_{t \to \infty} U_D(0,t) \uscat(x,t)
        \end{equation}
        exists in $H^1$.
        \item $\uwl$ is \emph{weakly bound} in the sense that
        \begin{equation}\label{eqn:uwl-weak-bound}
            \wlim_{t \to \infty} U_D(0,t) \uwl(t) = 0
        \end{equation}
    \end{enumerate}
\end{proposition}
\begin{proof}
    Define $\uscat(x,t) = U(t,0) \Omega_D \Omega^*_D u_0(x)$, $\uwl(x,t) = u(x,t) - \uscat(x,t)$.  Then, the decomposition~\eqref{eqn:u-redecomp} is trivially satisfied, and
    \begin{equation*}
        \slim_{t \to \infty} U(0,t) \uscat(x,t) = \Omega_D\Omega^*_D u_0
    \end{equation*}
    exists, so~\eqref{eqn:uscat-scats} holds.  Thus, it only remains to establish that $\uwl$ is weakly bound in the sense of~\eqref{eqn:uwl-weak-bound}.  To see why this is the case, we recall that $u_+ = \Omega_D^* u_0$, so
    \begin{equation}\label{eqn:asymp-equiv-uscat-mod-free-state}\begin{split}
        \lim_{t \to \infty} \lVert \uscat(x,t) - U_D(t,0) u_+ \rVert_{L^2} =& \lim_{t \to \infty} \lVert U(0,t) \uscat(x,t) - U(0,t) U_D(t,0) u_+ \rVert_{L^2}\\
        =& \lim_{t \to \infty} \lVert \Omega_D u_+ - U(0,t) U_D(t,0) u_+ \rVert_{L^2}\\
        =& 0
    \end{split}\end{equation}
    by the definition of $\Omega_D$.  Thus, for any $\phi \in L^2$, 
    \begin{equation*}\begin{split}
        \lim_{t \to \infty} \langle \phi, U_D(0,t) \uwl(t) \rangle =& \lim_{t \to \infty} \langle \phi, U_D(0,t) \urem(t) \rangle + \langle \phi, U_D(0,t) (\uwl(t) - \urem(t)) \rangle\\
        =& 0 + \lim_{t \to \infty} \langle \phi, U_D(0,t) (\uscat(t) - U_D(t,0) u_+) \rangle\\
        =& 0
    \end{split}\end{equation*}
    which verifies~\eqref{eqn:uwl-weak-bound}.
\end{proof}
In particular, we note that the decompositions~\eqref{eqn:main-decomp} and~\eqref{eqn:u-redecomp} are asymptotically equivalent:
\begin{corollary}\label[corollary]{cor:equiv-decomps}
    In the decompositions~\eqref{eqn:main-decomp} and~\eqref{eqn:u-redecomp}, 
    \begin{equation*}
        \lim_{t \to \infty} \lVert \uscat(x,t) - U_D(t,0) u_+(x) \rVert_{L^2} = 0
    \end{equation*}
    and 
    \begin{equation*}
        \lim_{t \to \infty} \lVert \uwl(x,t) -\urem(x,t) \rVert_{L^2} = 0
    \end{equation*}
\end{corollary}
Thus,~\eqref{eqn:urem-slow-spreading} will follow once we prove the following result:
\begin{theorem}\label{thm:uwl-sublin-spreading}
    If $\uwl(t)$ is \textit{weakly bound} in the sense that
    \begin{equation*}
        \wlim_{t \to \infty} U_D(0,t) \uwl(t) = 0
    \end{equation*}
    then for any $\epsilon > 0$,
    \begin{equation}\label{eqn:uwl-sublin-spread}
        \lim_{t \to \infty} \left\lVert F\left(\frac{|x|}{\epsilon t} \geq 1\right) \uwl(x,t) \right\rVert_{L^2_x} = 0
    \end{equation}
\end{theorem}

The proofs of~\Cref{thm:uwl-sublin-spreading,thm:better-rate} depend on the fact that weakly bound states have vanishing radial momentum.  We will formalize this idea in~\Cref{sec:slow-spread}, and show how it leads to a proof of~\Cref{thm:better-rate}.  We then show how to adapt the argument to prove~\Cref{thm:uwl-sublin-spreading} in~\Cref{sec:sublin-spread}.

\section{Slow spreading for initially localized data}\label{sec:slow-spread}

Let $\uwl$ be a weakly bound solution of~\eqref{eqn:main} satisfying the hypotheses of~\Cref{thm:better-rate}. Since we wish to prove that $\uwl$ spreads slowly, it would be natural to study the evolution of the quantity
\begin{equation*}
    \langle |x| \rangle_t = \int |x| |\uwl(x,t)|^2\;dx
\end{equation*}
(or in the case of~\Cref{thm:main-thm}, to study cut-off versions of this quantity).  The time derivative of $\langle |x| \rangle_t$ is given by
\begin{equation*}
    \frac{d}{dt} \langle |x| \rangle_t = 2\langle \gamma \rangle_t
\end{equation*}
where
\begin{equation*}
    \gamma = \frac{1}{2}[-i\Delta, |x|] = \frac{1}{2i} \left(\frac{x}{|x|} \cdot \nabla + \nabla \cdot \frac{x}{|x|}\right)
\end{equation*}
Intuitively, $\gamma$ measures the radial component of the momentum of the solution.  For time-independent potentials, nonscattering solutions must be bound states, so in particular they satisfy $\langle \gamma \rangle_t = 0$.  Results of this type do not appear to be feasible for time-dependent potentials $V$, since we have little control over what happens near the origin where $V$ is large.  However, it turns out to be possible to control radial momentum in exterior regions:
\begin{proposition}\label[proposition]{prop:rad-mom-vanishes}Let $\alpha > \beta$.  For any weakly bound solutions $\uwl$, 
    \begin{equation}
        \lim_{t \to \infty} \langle F_\alpha \gamma F_\alpha \rangle_t = 0
    \end{equation}
    where
    \begin{equation*}
        F_\alpha := F\left( \frac{|x|}{t^\alpha} \geq 1\right)
    \end{equation*}
    is a smooth cut-off to the region $|x| \gtrsim t^\alpha$.
\end{proposition}

To prove~\Cref{prop:rad-mom-vanishes}, we will show that
\begin{equation}\label{eqn:Omega-D-F-gamma-F}
    \Omega_{D}^{F_\alpha\gamma F_\alpha} := \slim_{t \to \infty} U_D(0,t) F_\alpha \gamma F_\alpha U(t,0)
\end{equation}
exists.  Since $U_D(0,t) \uwl \rightharpoonup 0$ in the weak $L^2$ topology, this implies that
\begin{equation*}\begin{split}
    \lim_{t \to \infty} \langle F_\alpha \gamma F_\alpha \rangle_t =& \lim_{t \to \infty} \langle U_D(0,t) \uwl(t), \Omega_D^{F_\alpha\gamma F_\alpha} \uwl(0) \rangle\\
    &+ \langle U_D(0,t) \uwl(t), U_D(0,t) F_\alpha\gamma F_\alpha \uwl(t) - \Omega_D^{F_\alpha\gamma F_\alpha} \uwl(0)\rangle\\
    =& 0
\end{split}\end{equation*}
To prove that~\eqref{eqn:Omega-D-F-gamma-F} exists, it suffices to prove the existence of the auxiliary wave operators
\begin{equation}\label{eqn:Omega-ext-F-gamma-F}
    \Omega_\ext^{F_\alpha \gamma F_\alpha} := \slim_{t \to \infty} U_\ext(0,t) F_\alpha \gamma F_\alpha U(t,0)
\end{equation}
and
\begin{equation}
    \Omega_{D,\ext} := \slim_{t \to \infty} U_D(0,t) F\left(4t^{1-\beta}|D| \geq 1\right) U_\ext(t,0)
\end{equation}
where $U_\ext(t,s)$ is the evolution semigroup for the equation
\begin{equation}\label{eqn:ext-schro}
    i\partial_t u + \Delta u = V_\ext(x,t) u
\end{equation}
since the composition rule for strong limits then shows that
\begin{equation*}\begin{split}
    \Omega_{D,\ext} \Omega_\ext^{F_\alpha \gamma F_\alpha} =& \slim_{t \to \infty} U_D(0,t) F\left(4t^{1-\beta}|D| \geq 1\right) F_\alpha \gamma F_\alpha U(t,0)\\
    =& \Omega_D^{F_\alpha \gamma F_\alpha} - \slim_{t \to \infty} U_D(0,t) F\left(4t^{1-\beta}|D| < 1\right) F_\alpha \gamma F_\alpha U(t,0)\\
    =& \Omega_D^{F_\alpha \gamma F_\alpha}
\end{split}\end{equation*}
where on the last line we have used that
\begin{equation*}\begin{split}
    F\left(t^{1-\beta}|D| < 1\right) F_\alpha \gamma F_\alpha =& F\left(4t^{1-\beta}|D| < 1\right) \gamma F_\alpha^2 + O(t^{-\alpha})\\
    =& O(t^{\beta-1}) + O(t^{-\alpha})
\end{split}\end{equation*}
vanishes as $t \to \infty$.

We will prove~\Cref{thm:better-rate} in three steps.  First, in~\Cref{sec:better-rate}, we will prove that~\Cref{prop:rad-mom-vanishes} implies~\Cref{thm:better-rate}.  We then prove the existence of $\Omega_{D,\ext}$ in~\Cref{sec:ext-D-WO}, and prove that $\Omega_\ext^{F_\alpha \gamma F_\alpha}$ exists in~\Cref{sec:ext-F-gamma-F-WO}.

\subsection{Proof of~\Cref{thm:better-rate}}\label{sec:better-rate}
Assuming~\Cref{prop:rad-mom-vanishes}, then~\Cref{thm:better-rate} follows from the following theorem:
\begin{theorem}\label{thm:optimal-spread-theorem}
     Let $\alpha \geq \alpha_0 := \frac{2}{2+\mu}$.  Suppose that $u(t)$ is a solution to~\eqref{eqn:main} such that
    \begin{equation}\label{eqn:OST-hypo-1}
        \lVert |x|^{1/2} u(1) \rVert_{L^2} < \infty
    \end{equation}
    and
    \begin{equation} \label{eqn:OST-hypo-2}
        \lim_{t \to \infty} \langle F_\alpha \gamma F_\alpha \rangle_t = 0
    \end{equation}
    Then, for $t \geq 1$,
    \begin{equation*}
        \langle |x| \rangle_t \lesssim_u t^{\alpha}
    \end{equation*}
\end{theorem}
\begin{remark}
    The range $\alpha \geq \alpha_0$ in the hypothesis for~\Cref{thm:optimal-spread-theorem} is larger than the range $\alpha > \beta > \frac{3}{2+2\mu}$ allowed by~\Cref{prop:rad-mom-vanishes}.  We expect that some version of~\Cref{prop:rad-mom-vanishes} also holds for $\alpha > \alpha_0$ (and possibly for $\alpha = \alpha_0$), at which point~\Cref{thm:optimal-spread-theorem} would immediately give improvements to the rate in~\Cref{thm:better-rate} (C.f.~\Cref{rmk:rate-improvements}).
\end{remark}
\begin{proof}
    The point here is that
    \begin{equation*}
        \lVert |x|^{1/2} u(x,t) \rVert_{L^2}^2 = \langle |x| \rangle_t = \langle F_\alpha |x| F_\alpha \rangle_t + \langle (1 - F_\alpha^2) |x| \rangle_t
    \end{equation*}
    From the definition of $F_\alpha$, we see at once that
    \begin{equation*}
        \langle (1 - F_\alpha^2) |x| \rangle_t \lesssim t^{\alpha}
    \end{equation*}
    so it is sufficient to prove that $\langle F_\alpha |x| F_\alpha \rangle_t$ grows at worst like $t^\alpha$.  Differentiating, we see that
    \begin{equation*}\begin{split}
        \frac{d}{dt} \langle F_\alpha |x| F_\alpha \rangle_t =& 2\langle F_\alpha \gamma F_\alpha \rangle_t- 2\alpha \langle \frac{|x|^2}{t^{\alpha + 1}} F'_\alpha F_\alpha \rangle_t
    \end{split}\end{equation*}
    Since $F$ is increasing monotonically, the second term here is manifestly negative (it represents the mass that is `lost' as the boundary of the cut-off expands outward), and we will be finished if we can show that
    \begin{equation*}
        \int_1^t \langle F_\alpha \gamma F_\alpha \rangle_s\;ds \lesssim t^\alpha
    \end{equation*}
    To see that this is the case, we will need to take a further derivative:
    \begin{equation}\label{eqn:F-gamma-F-Heis}\begin{split}
        \frac{d}{dt} \langle F_\alpha \gamma F_\alpha \rangle_t =& \left\langle F_\alpha [-i\Delta, \gamma] F_\alpha \right\rangle_t + \left\langle [-i\Delta, F_\alpha] \gamma F_\alpha + F_\alpha \gamma [-i\Delta, F_\alpha]\right\rangle_t\\
        &+ \left\langle F_\alpha [iV, \gamma] F_\alpha \right\rangle_t -\alpha \left\langle \frac{|x|}{t^{\alpha+1}} F'_\alpha \gamma F_\alpha + F_\alpha \gamma F'_\alpha \frac{|x|}{t^{\alpha+1}}\right\rangle_t
    \end{split}\end{equation}
    The first term is positive semi-definite:
    \begin{equation}\label{eqn:wls-symb-comm}
        \left\langle F_\alpha [-i\Delta, \gamma] F_\alpha \right\rangle_t = -2 \left\langle F_\alpha \nabla_x \cdot S(x) \nabla_x F_\alpha \right\rangle_t
    \end{equation}
    where
    \begin{equation*}
        S(x) = \frac{1}{|x|} \left(I - \frac{x}{|x|} \otimes \frac{x}{|x|}\right)
    \end{equation*}
    For the second term, we use~\Cref{lem:three-term-sym} to symmetrize:
    \begin{equation}\label{eqn:wls-bdy-comm}\begin{split}
        \langle [-i\Delta, F_\alpha] \gamma F_\alpha + F_\alpha \gamma [-i\Delta, F_\alpha]\rangle_t =& 2 t^{-\alpha} \langle F'_\alpha \gamma^2 F_\alpha + F_\alpha \gamma^2 F'_\alpha \rangle_t -i t^{-2\alpha} \langle F''_\alpha \gamma F_\alpha - F_\alpha \gamma F''_\alpha\rangle_t \\&
        -i t^{-\alpha} \left\langle \frac{2}{|x|}F'_\alpha \gamma F_\alpha - F_\alpha \gamma \frac{2}{|x|}F'_\alpha\right\rangle_t \\
        =& 4 t^{-\alpha} \langle \sqrt{F_\alpha F'_\alpha} \gamma^2 \sqrt{F_\alpha F'_\alpha} \rangle_t + O(t^{-3\alpha})
    \end{split}\end{equation}
    For the final term, we symmetrize and use Cauchy-Schwarz to find that
    \begin{equation}\label{eqn:wls-time-deriv}\begin{split}
        \left|\left\langle \frac{|x|}{t^{\alpha+1}} F'_\alpha \gamma F_\alpha + F_\alpha \gamma F'_\alpha \frac{|x|}{t^{\alpha+1}}\right\rangle_t \right| =& 2\left|\left\langle \frac{|x|}{t^{\alpha+1}} \sqrt{F_\alpha F'_\alpha} \gamma \sqrt{F_\alpha F'_\alpha}\right\rangle_t\right|\\
        \leq& 2 t^{-\alpha} \langle \sqrt{F_\alpha F'_\alpha} \gamma^2 \sqrt{F_\alpha F'_\alpha} \rangle_t + 8 \left\lVert \frac{|x|}{t^{\alpha/2+1}} \sqrt{F_\alpha F'_\alpha} u(x,t)\right\rVert_{L^2}^2\\
        \leq& 2 t^{-\alpha}\langle \sqrt{F_\alpha F'_\alpha} \gamma^2 \sqrt{F_\alpha F'_\alpha} \rangle_t + O(t^{\alpha -2})
    \end{split}\end{equation}
    Finally, the term involving the commutator with the potential has the slowest decay:
    \begin{equation}\label{eqn:wls-pot-comm}
        \left\langle F_\alpha [iV, \gamma] F_\alpha \right\rangle_t = O(t^{-(1+\mu)\alpha})
    \end{equation}
    Now, for $\alpha \geq \alpha_0 = \frac{2}{2+\mu}$, we have that $\alpha - 2 \geq  -(1+\mu)\alpha$.  Thus, after using the positive term in~\eqref{eqn:wls-bdy-comm} to cancel any negative contribution from~\eqref{eqn:wls-time-deriv}, we see that
    \begin{equation*}
        \frac{d}{dt} \langle F_\alpha \gamma F_\alpha \rangle_t \geq O(t^{\alpha-2})
    \end{equation*}
    Since $\lim_{t \to \infty} \langle F_\alpha \gamma F_\alpha \rangle_t = 0$, we can integrate to find that
    \begin{equation*}
        \max(0, \langle F_\alpha \gamma F_\alpha \rangle_t) \lesssim t^{\alpha - 1}
    \end{equation*}
    (Note that this is only a one-sided inequality: the negative part of $\langle F_\alpha \gamma F_\alpha \rangle_t$ also decays to $0$, but in general need not do so at a polynomial rate.)  Integrating once more now gives the slow spreading required in~\eqref{eqn:uwl-better-rate}, completing the proof of~\Cref{thm:better-rate}.
\end{proof}

For later use, we end this section by repackaging the above estimate on $\langle F_\alpha \gamma F_\alpha \rangle_t$ as a propagation estimate:
\begin{corollary}\label[corollary]{cor:F-gamma-F-PRES}
    For any $t_0 < t_1$ and any solution $u$ to~\eqref{eqn:main},
    \begin{equation}\label{eqn:F-gamma-F-PRES}
        \lVert \gamma \sqrt{F_\alpha F_\alpha'} u(x,t) \rVert_{L^2(t_0, t_1; L^2_x)}^2 = \langle F_\alpha \gamma F_\alpha \rangle_{t_1} - \langle F_\alpha \gamma F_\alpha \rangle_{t_0} + O(t_0^{\alpha - 1}\lVert u \rVert_{L^\infty_tL^2_x})
    \end{equation}
    In particular,
    \begin{equation}\label{eqn:F-gamma-F-PRES-inf}
        \lVert \gamma \sqrt{F_\alpha F_\alpha'} u(x,t) \rVert_{L^2_t(t_0, \infty; L^2_x)}^2 \lesssim_{t_0} \lVert u \rVert_{L^\infty_t H^1_x}^2
    \end{equation}
\end{corollary}

\subsection{Existence of the exterior-to-Dollard wave operator}\label{sec:ext-D-WO}

We now prove the existence of the exterior-to-Dollard wave operator $\Omega_{D,\ext}$.  Intuitively, the main idea is that near the microlocal propagation set $x = 2pt + \{\textup{error}\}$ for the free Schr\"odinger equation, the exterior potential $\Vext(x,t)$ and the Dollard potential $V(2pt,t)$ are approximately equal provided that the error term does not grow too quickly in time.  In the prior works~\cite{Soffer-Wu,sofferScatteringLocalizedStates2024a,liu2025large,stewartWeaklyLocalizedStates2025} that constructed free channel wave operators for Schr\"odinger equations, the slow growth of the error term was imposed through the definition of the free channel.  In our setting, we must prove that all states for the Schr\"odinger equation with the exterior potential $\Vext$ have dynamics given asymptotically by $U_D$, so it is not possible to make such a strong assumption.  Instead, we will first prove that we can prove the slow growth of the error on a dense set of states, and then use Cook's method to construct $\Omega_{D,\ext}$.

\subsubsection{Slow growth of the error term}

Slow growth of the error term is given quantitatively in the following theorem:
\begin{proposition}\label[proposition]{prop:ESG}
    Let $u$ be a solution to 
    \begin{equation}\label{eqn:ext-Schro}
        i\partial_t u + \Delta u = \Vext(x,t) u
    \end{equation}
    If $|x-2p| u(x,1) \in L^2_x$, then 
    \begin{equation}\label{eqn:error-slow-growth}
        \lVert |x - 2pt| u(x,t) \rVert_{L^2_x} \lesssim_{\rho, u} \jBra{t}^{\rho}
    \end{equation}
    for $\rho > 2 - \beta(1+\mu)$ (c.f.~\Cref{tab:constants}).
\end{proposition}
\begin{proof}
    We will obtain~\eqref{eqn:error-slow-growth} by performing a propagation estimate against the observable 
    $$A(t) = \frac{|x-2pt|^2}{t^{2\rho}}$$
    Taking the Heisenberg derivative and observing that
    \begin{equation*}
        \frac{d}{dt} |x-2pt|^2 + [-i\Delta, |x-2pt|^2] = 0
    \end{equation*}
    we find that
    \begin{subequations}\begin{align}
        \frac{d}{dt} \left\langle \frac{|x-2pt|^2}{t^{2\rho}}\right\rangle_t =& -2\rho\left\langle \frac{|x-2pt|^2}{t^{2\rho+1}} \right\rangle \label{eqn:ESG-good-sign}\\
        &+ \left\langle \left[i\Vext, \frac{|x-2pt|^2}{t^{2\rho}}\right]\right\rangle_t \label{eqn:ESG-pot-comm}
    \end{align}
    \end{subequations}
    The first term~\eqref{eqn:ESG-good-sign} has a favorable sign.  For the second term, we see that
    \begin{equation*}
        |x-2pt|^2 = x^2 -2pt\cdot x -2x \cdot pt + 4p^2t^2
    \end{equation*}
    so
    \begin{equation*}\begin{split}
        [i\Vext, |x-2pt|^2] =& 4t x \cdot \nabla_x \Vext - 4pt^2 \cdot \nabla_x \Vext - 4t^2 \nabla_x \Vext \cdot p\\
        =& 2t(x-2pt) \cdot \nabla_x \Vext + 2t \nabla_x \Vext \cdot (x-2pt)
    \end{split}\end{equation*}
    In particular, using Cauchy-Schwarz and Cauchy's inequalities yields that
    \begin{equation*}
        |\eqref{eqn:ESG-pot-comm}| \leq \frac{2 \rho}{t^{2\rho+1}} \lVert |x-2pt| u(x,t) \rVert_{L^2_x}^2 + C_\rho t^{3-2\rho} \lVert \nabla_x \Vext(x,t) u(x,t) \rVert_{L^2_x}^2
    \end{equation*}
    The first term can be controlled using the good-sign term~\eqref{eqn:ESG-good-sign}, while the second term is perturbative, having size $O(t^{3-2\rho -2\beta(1+\mu)})$.  Based on our assumption that $\rho > 2 - (1+\mu)\beta$, we see that $3-2\rho - 2\beta(1+\mu) < -1$, and the last error term is integrable.  Integrating now gives that
    \begin{equation*}
        \left \langle \frac{|x-2pt|^2}{t^{2\rho}} \right\rangle_t - \langle |x-2p|^2 \rangle_1 \lesssim 1
    \end{equation*}
    which implies~\eqref{eqn:error-slow-growth}.
\end{proof}

\begin{remark}\label{remark:rho-bound}
    The lower bound $\rho > 2 - (1+\mu)\beta$ in the hypothesis for~\Cref{prop:ESG} is natural: The Dollard correction to the free trajectory is
    \begin{equation*}
        \int_1^t \int_s^\infty \nabla_x\Vext(2p\tau, \tau)\;d\tau ds
    \end{equation*}
    which has size $O(t^{2-(1+\mu)\beta})$.
\end{remark}

A similar result also holds for higher moments:

\begin{proposition}\label[proposition]{prop:ESG-HO}
    Let $n \geq 1$.  If $u$ solves~\eqref{eqn:ext-Schro} and $|x - 2p|^n u(x,1) \in L^2_x$, then
    \begin{equation}\label{eqn:error-slow-growth-HO}
        \lVert |x - 2pt|^n u(x,t) \rVert_{L^2_x} \lesssim_{\rho, u} \jBra{t}^{n\rho}
    \end{equation}
    for $\rho > 2-(1+\mu)\beta$.
\end{proposition}
\begin{proof}
    We will proceed by (strong) induction on $n$.  The base case was established in~\Cref{prop:ESG}, so we can assume that $n > 1$.  Taking the Heisenberg derivative of the observable $A(t) = \frac{|x-2pt|^{2n}}{t^{2n\rho}}$, we find that
    \begin{equation*}\begin{split}
        \frac{d}{dt} \left\langle \frac{|x-2pt|^{2n}}{t^{2n\rho}} \right\rangle_t = -2n \left\langle \frac{|x-2pt|^{2n}}{t^{2n\rho+1}} \right\rangle_t + \left\langle \left[i\Vext, \frac{|x-2pt|^{2n}}{t^{2n\rho}}\right] \right\rangle_t
    \end{split}\end{equation*}
    As before, the first term has a good sign, and will be used to compensate certain unacceptably large contributions from the second term, which is sign-indefinite.  Turning to the second term, we see that
    \begin{equation*}\begin{split}
        [i\Vext, |x-2pt|^{2n}] =& \sum_{j=0}^{n-1} |x-2pt|^{2j} [i\Vext, |x-2pt|^2] |x-2pt|^{2(n-j-1)}\\
        =& 2t\sum_{j=0}^{n-1} |x-2pt|^{2j} \left( (x-2pt) \cdot \nabla_x \Vext + \nabla_x \Vext \cdot (x-2pt)\right) |x-2pt|^{2(n-j-1)}
    \end{split}\end{equation*}
    By repeatedly commuting $|x-2pt|^2$ weights from the side with more than $n$ weights, we can rewrite this expression as
    \begin{equation}\label{eqn:Vext-pseudoconf-comm}\begin{split}
        [\Vext, |x-2pt|^{2n}] =& 2nt |x-2pt|^{m} (x-2pt) \cdot \nabla_x \Vext |x-2pt|^{2n-m-2}\\
        &+|x-2pt|^{2n-m-2}  \nabla_x \Vext \cdot (x-2pt) |x-2pt|^{m}\\
        &+ R_n
    \end{split}\end{equation}
    where
    \begin{equation*}
        m = \begin{cases}
            n-1 & n \text{ odd}\\
            n-2 & n \text{ even}
        \end{cases}
    \end{equation*}
    and $R_n$ is a remainder involving higher order commutator terms.  By repeatedly applying this symmetrization argument on the commutator, we see that we can  write $R_n$ as a sum of terms of the form
    \begin{equation*}
        C_{a,b,c,d} t^{b+c} |x-2pt|^{2a} (x-2pt)^{\otimes b} \cdot D_x^{b+c} \Vext \cdot (x-2pt)^{\otimes c} |x-2pt|^{2d}
    \end{equation*}
    where
    \begin{equation}\label{eqn:S-a-b-c-d-n}
        a+b+c+d = n,\qquad b+c > 1, \qquad 2a+b = n \text{ or } c + 2d = n
    \end{equation}
    We now estimate the terms arising in~\eqref{eqn:Vext-pseudoconf-comm}.  For the main term, we can use the inductive hypothesis to estimate
    \begin{equation}\label{eqn:ESG-HO-main}\begin{split}
        \left\langle\frac{[\Vext, |x-2pt|^{2n}] - R_n}{t^{2n\rho}}\right\rangle_t \lesssim_n& t^{1-2n\rho} \lVert \nabla_x \Vext \rVert_{L^\infty_x} \lVert |x-2pt|^n u \rVert_{L^2_x}  \lVert  |x-2pt|^{n-1} u \rVert_{L^2_x}\\
        \lesssim_{n,u}& t^{1-2\rho -\beta(1+\mu)}\lVert |x-2pt|^n u \rVert_{L^2_x}\\
        \leq& \epsilon \frac{\lVert |x-2pt|^n u \rVert_{L^2_x}^2}{t^{2n\rho+1}} +  C_{n,u}\epsilon^{-1} t^{3 - 2\rho - 2 \beta(1+\mu)}
    \end{split}\end{equation}
Since $\rho > 2 - (1+\mu)\beta$, this term is integrable in time.  Similarly, $R_n$ can be estimated as
    \begin{equation}\label{eqn:ESG-HO-rem}\begin{split}
        \frac{|\langle R_n \rangle_t|}{t^{2n\rho}} \lesssim&{}_n t^{-2n\rho} \sum_{(a,b,c,d) \in \mathcal{I}(n)}t^{b+c}\lVert D_x^{b+c} \Vext \rVert_{L^\infty} \lVert |x-2pt|^{n} u \rVert_{L^2_x} \lVert |x-2pt|^{c+2d} u \rVert_{L^2_x}\\
        &+ \{\textup{similar terms}\}\\
        \lesssim&{}_{u,\rho,n} \sum_{(a,b,c,d) \in \mathcal{I}(n)}t^{(b+c)(1-\rho-\beta)}t^{-n\rho}\lVert |x-2pt|^{n} u \rVert_{L^2_x} t^{-\mu\beta} + \{\textup{similar terms}\} \\
        \leq& \epsilon \frac{\lVert |x-2pt|^n u \rVert_{L^2_x}^2}{t^{2n\rho+1}} + C_{u,\rho,n} \epsilon^{-1} \sum_{(a,b,c,d) \in \mathcal{I}(n)}t^{2(b+c)(1-\rho -\beta)} t^{1-2\mu\beta}
    \end{split}
    \end{equation}
    with $\mathcal{I}(n)$ the set of all $(a,b,c,d)$ satisfying~\eqref{eqn:S-a-b-c-d-n} with $2a + b = n$, and $\{\textup{similar terms}\}$ denotes an analogous sum over the terms where $c + 2d = n$. Here we also used $b+2c+2d=n$ when $2a+b=n$. Now, for $\beta < 1 \leq \frac{1}{\mu}$, the hypothesis $\rho > 2 - (1+\mu)\beta$ implies that $\rho > 1 - \beta$, so~\eqref{eqn:ESG-HO-rem} is lower order.  In particular, by choosing $\epsilon$ sufficiently small in~\eqref{eqn:ESG-HO-main} and~\eqref{eqn:ESG-HO-rem}, we find that
    \begin{equation*}
        \left\langle \frac{|x-2pt|^{2n}}{t^{2n\rho}} \right\rangle_t - \left\langle |x-2p|^{2n} \right\rangle_1 \lesssim_{n,\rho,u} 1
    \end{equation*}
    as required.
\end{proof}

\subsubsection{Cook's method set-up}

With~\Cref{prop:ESG,prop:ESG-HO} in hand, we are now prepared to prove the existence of
\begin{equation*}
    \Omega_{D,\ext} = \slim_{t \to \infty} U_D(0,t) F\left(4t^{1-\beta} |D| \geq 1\right) U_\ext(t,0)
\end{equation*}
By Cook's method and the support properties of $\Vext$, it suffices to prove that the integral
\begin{equation*}
    \int_1^\infty U_D(0,t)F\left(4t^{1-\beta} |D| \geq 1\right)\left(\Vext(2pt,t) - \Vext(x,t)\right) U_\ext(t,0) \phi \;dt
\end{equation*}
converges in $L^2$ for all $\phi$ in some dense subset of $L^2$.  Specifically, defining $\psi(t) = U_\ext(t,0) \phi$, we will assume that
\begin{equation*}
    \lVert |x-2p|^n \psi(x,1) \rVert_{L^2_x} < \infty,\qquad n = 1,2,\cdots, N
\end{equation*}
for some large but fixed $N$.  \Cref{prop:ESG} then guarantees that
\begin{equation}\label{eqn:psi-moment-cond}
    \lVert |x-2pt|^n \psi(x,t) \rVert_{L^2_x} \lesssim_\phi \jBra{t}^{n\rho}
\end{equation}
for $n = 1,2,\cdots, N$, where $\rho \in (1-\beta, 1/2)$ (c.f.~\Cref{tab:constants}).  To take advantage of~\eqref{eqn:psi-moment-cond}, we will write
\begin{equation*}
    f(x,t) = e^{-it\Delta}\psi(x,t)
\end{equation*}
since then~\eqref{eqn:psi-moment-cond} reduces to
\begin{equation}\label{eqn:f-moment-cond}
    \lVert \jBra{x}^n f(x,t) \rVert_{L^2_x} \lesssim_\phi \jBra{t}^{n\rho}
\end{equation}
and we can write
\begin{equation}\label{eqn:V-diff-f}\begin{split}
    (\Vext(2pt,t) - \Vext(x,t)) \psi    =& (\Vext(2pt,t) - \Vext(x,t)) e^{it\Delta} f\\
                                        =& \frac{1}{(2\pi)^{3/2}} \int (\Vext(2\xi t, t) - \Vext(x,t)) e^{ix\cdot \xi - it\xi^2} \hat{f}\;d\xi
\end{split}\end{equation}
This appears similar to our starting point in~\Cref{sec:exterior-bounds}.  However, the bounds we obtain for $\hat{f}(\cdot,t)$ in $L^\infty_\xi$ grow as $t \to \infty$, so we cannot simply mimic the argument we used to construct the wave operator.  Instead,  we use a higher-order Taylor expansion to write
\begin{equation*}\begin{split}
    \Vext(2\xi t, t) - \Vext(x,t) =& \nabla_x \Vext(x,t) \cdot (2\xi t- x) + \frac{1}{2} (2\xi t - x) \cdot D^2_x \Vext(x,t) (2\xi t-x)\\
    &+ (2\xi t- x)^{\otimes 3} \cdot \CVext^3(x,\xi,t)
\end{split}\end{equation*}
where
\begin{equation}\label{eqn:CVext-3-def}
    \CVext^3(x,\xi,t) = \int_0^1 \int_0^{\lambda_1} \int_0^{\lambda_2} D_x^3 \Vext(\lambda x + 2(1-\lambda)t\xi) \;d\lambda d\lambda_2 d\lambda_1
\end{equation}
Substituting this into~\eqref{eqn:V-diff-f}, we find that
\begin{equation*}\begin{split}
    (\Vext(2pt,t) - \Vext(x,t)) \psi =& \frac{i\nabla_x \Vext(x,t)}{(2\pi)^{3/2}}\cdot  \int  \nabla_\xi e^{ix\cdot \xi - it\xi^2} \hat{f}\;d\xi - \frac{D^2_x \Vext(x,t)}{(2\pi)^{3/2}}\cdot  \int  D^2_\xi e^{ix\cdot \xi - it\xi^2} \hat{f}\;d\xi\\
    &- \frac{i}{(2\pi)^{3/2}} \int  D_\xi^3 e^{ix\cdot \xi - it\xi^2} \cdot \CVext^3(x,\xi,t) \hat{f}\;d\xi
\end{split}
\end{equation*}
Both terms on the first line can be controlled in $L^2_x$ using Plancherel:
\begin{equation}\label{eqn:HO-Dollard-lin-term}\begin{split}
    \left\lVert \nabla_x \Vext(x,t) \cdot  \int  \nabla_\xi e^{ix\cdot \xi - it\xi^2} \hat{f}\;d\xi\right\rVert_{L^2_x} =& \left\lVert \nabla_x \Vext(x,t) \cdot  \int  e^{ix\cdot \xi - it\xi^2} \nabla_\xi \hat{f}\;d\xi\right\rVert_{L^2_x}\\
    \leq& \lVert \nabla_x \Vext(x,t) \rVert_{L^\infty_x} \lVert x f \rVert_{L^2}\\
    \lesssim& t^{\rho -(1+\mu)\beta} 
\end{split}\end{equation}
and
\begin{equation}\label{eqn:HO-Dollard-quad-term}\begin{split}
    \left\lVert D^2_x \Vext(x,t) \cdot  \int  D^2_\xi e^{ix\cdot \xi - it\xi^2} \hat{f}\;d\xi\right\rVert_{L^2_x} =& \left\lVert D^2_x \Vext(x,t) \cdot  \int  e^{ix\cdot \xi - it\xi^2} D^2_\xi \hat{f}\;d\xi\right\rVert_{L^2_x}\\
    \leq& \lVert D^2_x \Vext(x,t) \rVert_{L^\infty_x} \lVert |x|^2 f \rVert_{L^2}\\
    \lesssim& t^{2\rho -(2+\mu)\beta} 
\end{split}\end{equation}
Based on our assumptions that $\rho < 1/2$, $\beta > \frac{3}{2(1+\mu)}$, and $\mu \in (1/2, 1]$, both of these terms are integrable, so it only remains to prove that the last term is integrable.  Here, integration by parts yields
\begin{equation*}\begin{split}
    \int D_\xi^3 e^{ix\cdot \xi - it\xi^2} \cdot \CVext^3(x,\xi,t) \hat{f}\;d\xi =& \int e^{ix\cdot \xi - it\xi^2} D_\xi^3 \cdot \left(\CVext^3(x,\xi,t) \hat{f}\right) \;d\xi\\
    =& \int e^{ix\cdot \xi - it\xi^2} F(t^{1/2}|\xi -\xi_\textup{stat}| \leq 1) D_\xi^3 \cdot \left(\CVext^3(x,\xi,t) \hat{f}\right) \;d\xi\\
    &+ \int e^{ix\cdot \xi - it\xi^2} F(t^{1/2}|\xi -\xi_\textup{stat}| \geq 1) D_\xi^3 \cdot \left(\CVext^3(x,\xi,t) \hat{f}\right) \;d\xi\\
    =:& \trmIst + \trmInst
\end{split}\end{equation*}
where we have again distinguished the near stationary and non-stationary parts of the integral.  Based on the previous reasoning, the existence of the exterior-to-Dollard wave operator will follow once we prove that $\lVert \trmIst(x,t) \rVert_{L^2_x}$ and $\lVert \trmInst(x,t) \rVert_{L^2_x}$ are integrable in time.

\subsubsection{Bounds for the near stationary term}
For the near-stationary term, we observe that
\begin{equation*}
    \lambda x + 2(1-\lambda)t\xi = x + (\lambda-1)(x-2t\xi)
\end{equation*}
Since $x = 2t \xi_\textup{stat}$, we see that if $|x| \gtrsim t^\beta$, then on the support of $\trmIst$,
\begin{equation*}
    |D^j_\xi \CVext^3(x,\xi,t)| \lesssim \jBra{x}^{-3-\mu}, \qquad 0 \leq j \leq 5
\end{equation*}
while for $|x| \lesssim t^\beta$, 
\begin{equation*}
    |D^j_\xi \CVext^3(x,\xi,t)| \lesssim t^{-(3+\mu)\beta}, \qquad 0 \leq j \leq 5
\end{equation*}
(c.f.~\Cref{lem:CVext-decay-near-stat}).  In particular, we see that
\begin{equation}\label{eqn:trmIst-bound}\begin{split}
    \lVert \trmIst \rVert_{L^2_x} \lesssim&\lVert \max(t^\beta, \jBra{x})^{-3-\mu} \rVert_{L^2_x} \lVert F(t^{1/2}|\xi -\xi_\textup{stat}| \leq 1) \jBra{D_\xi}^3 \hat{f} \rVert_{L^1_\xi} \\    
    \lesssim&t^{-(3/2+\mu)\beta} \lVert F(t^{1/2}|\xi -\xi_\textup{stat}| \leq 1) \jBra{D_\xi}^3 \hat{f} \rVert_{L^1_\xi}\\
    \lesssim& t^{-3/4 - (3/2+\mu) \beta} \lVert  \jBra{x}^3 f \rVert_{L^2}\\
    \lesssim& t^{3\rho - 3/4 - (3/2 + \mu) \beta}
\end{split}\end{equation}
and under our assumptions on $\rho$ and $\beta$, $3\rho - 3/4 - (3/2 + \mu) \beta < -1$.

\subsubsection{Bounds for the non-stationary term}
Turning to the nonstationary term, we first introduce a dyadic decomposition
\begin{equation*}
    \trmInst^j := \int e^{ix\cdot \xi - it\xi^2} \chi_j(t^{1/2}|\xi -\xi_\textup{stat}|) D_\xi^3 \cdot \left(\CVext^3(x,\xi,t) \hat{f}\right) \;d\xi
\end{equation*}
so
\begin{equation*}
    \trmInst = \sum_{j=0}^\infty \trmInst^j
\end{equation*}
Considering one summand, we observe that the differential operator
\begin{equation*}
    \mathcal{L}_0 = \frac{(x - 2t\xi)}{i|x-2t\xi|^2} \cdot \nabla_\xi
\end{equation*}
satisfies
\begin{equation*}
    \mathcal{L}_0 e^{ix\cdot\xi - it\xi^2} = e^{ix\cdot\xi - it\xi^2}
\end{equation*}
so integrating by parts twice gives
\begin{equation*}\begin{split}
    \trmInst^j = \int e^{ix\cdot\xi -it\xi^2} \bigl(\mathcal{L}_0^*\bigr)^2 \left(\chi_j(t^{1/2}|\xi -\xi_\textup{stat}|) D_\xi^3 \cdot \left(\CVext^3(x,\xi,t) \hat{f}\right)\right)\;d\xi
\end{split}\end{equation*}
By modifying the argument for~\Cref{lem:CVext-nonstat-lemma}, we obtain the following result on $\CVext^3$:
\begin{lemma}\label[lemma]{lem:CVext-3-exterior-bound}
    Let $\xi \in \supp \chi_j(t^{1/2} |\xi - \xi_\textup{stat}|)$.  If $|x| > 2^{j+10}t^{1/2}$, then for $0 \leq a \leq 5$:
    \begin{equation}\label{eqn:CVext-3-far-field-bound}
        |D_\xi^a \CVext^3(x,\xi,t)| \lesssim \jBra{x}^{-(3+\mu)}
    \end{equation}
    Otherwise, we have the unconditional bound
    \begin{equation}\label{eqn:CVext-3-uncond}
        |D_\xi^a \CVext^3(x,\xi,t)| \lesssim t^{-(3+\mu)\beta}
    \end{equation}
\end{lemma}
\begin{proof}
    The unconditional bound~\eqref{eqn:CVext-3-uncond} follows from the definition of $\CVext$, while the bound~\eqref{eqn:CVext-3-far-field-bound} follows from observing that
    \begin{equation*}
        |x - 2t\xi| = 2t|\xi - \xi_\textup{stat}| \lesssim 2^j t^{1/2}
    \end{equation*}
    and arguing as in the proof of~\Cref{lem:CVext-nonstat-lemma}.
\end{proof}
Now, in the expression
\begin{equation*}
    \bigl(\mathcal{L}_0^*\bigr)^2 \left(\chi_j(t^{1/2}|\xi -\xi_\textup{stat}|) D_\xi^3 \cdot \left(\CVext^3(x,\xi,t) \hat{f}\right)\right)
\end{equation*}
each derivative produces a different contribution depending on which function it acts on:
\begin{itemize}
    \item A derivative that hits $\hat{f}$ will turn into a space weight (and thus contribute a factor of $t^\rho$ in the estimate),
    \item Any derivative that hits $\chi_j(t^{1/2}|\xi -\xi_\textup{stat}|)$ or a $\frac{(x - 2t\xi)}{|x-2t\xi|^2}$ factor will contribute a factor of $O(t^{1/2} 2^{-j})$,
    \item The size of the term $\CVext^3$ is unaffected by differentiation.
\end{itemize}
Taking these considerations into account, we see that for $|x| > 2^{j+10}t^{1/2}$:
\begin{equation*}\begin{split}
    |\trmInst^j(x,t)| \lesssim& 2^{-j/2}t^{-7/4} \max(\jBra{x}, t^\beta)^{-(3+\mu)} \left\lVert \jBra{x + t^{1/2} 2^{-j}}^2\jBra{x}^3 f \right\rVert_{L^2_x}
\end{split}\end{equation*}
In particular, taking the $L^2$ integral in $x$ gives that
\begin{equation}\label{eqn:trmInst-ext}
    \lVert \trmInst^j(x,t) \rVert_{L^2(|x| \geq 2^{j+10}t^{1/2})} \lesssim 2^{-j/2} t^{-7/4} t^{-(3/2+\mu)\beta} \left(2^{-2j} t + t^{2\rho}\right)t^{3\rho}
\end{equation}
which is integrable in time after summing over $j \geq 0$, as required.  Thus, it only remains to consider the contribution from the region $|x| \leq 2^{j+10}t^{1/2}$.  The difficulty here is that the unconditional bound~\eqref{eqn:CVext-3-uncond} does not produce bounds that can be summed in $j$.  To overcome this, we need additional refinements to~\Cref{lem:CVext-3-exterior-bound}:
\begin{lemma}\label[lemma]{lem:CVext-3-exterior-bound-refined}
    Suppose that $|x| \leq t^{1/2} 2^{j+10}$ and $t^{-1/2} 2^j \leq |\xi - \xi_\textup{stat}| \leq t^{-1/2} 2^{j+2}$.  Then, we have the refined unconditional bound
    \begin{equation}\label{eqn:CVext-3-uncond-refined}
        |D_\xi^a \CVext^3(x,\xi,t)| \lesssim t^{-(2+\mu)\beta-1/2} 2^{-j}
    \end{equation}
    for $a \in [0,5]$.
    
    Moreover, if
    \begin{equation*}
        \theta = \arccos \frac{x \cdot (2t\xi - x)}{|x||x-2t\xi|}
    \end{equation*}
    (see~\Cref{fig:CVext3-ref-diagram-1}) then we have the refined angular bound
    \begin{equation}\label{eqn:CVext-3-angular-refined}
        |D_\xi^a \CVext^3(x,\xi,t)| \lesssim (|x|\sin \theta)^{-(2+\mu)} t^{-1/2} 2^{-j}
    \end{equation}
    for $a \in [0,5]$.
\end{lemma}
\begin{figure}
  \centering
  \begin{tikzpicture}[scale=1, >=Stealth]
    % Coordinates
    \coordinate (O) at (0,0);
    \coordinate (X) at (-3, 1.5);
    \coordinate (Z) at (5, 3.5);
    
    % Main Triangle Lines
    \draw[thick] (O) -- (X) node[left, xshift=-5pt, yshift=5pt] {$x$};
    \draw[thick] (O) -- (Z) node[right, xshift=5pt, yshift=5pt] {$2t\xi$};
    \draw[thick] (X) -- (Z);
    \draw (O) node[below, xshift=5pt, yshift=-3pt] {$0$};
    
    % Dots at vertices
    \fill (X) circle (1.5pt);
    \fill (Z) circle (1.5pt);
    \fill (O) circle (1.5pt);
    
    % Line segment R (perpendicular to X-Z from Origin)
    \coordinate (P) at ($(X)!(O)!(Z)$);
    \draw[thick] (O) -- (P);
    \fill (P) circle (1.5pt) node[above, xshift=-5pt, yshift=5pt] {$z(\lambda_0)$};
    
    % Right angle symbol at P
    \draw ($(P)!0.2cm!(X)$) -- ($($(P)!0.2cm!(X)$)!0.2cm!90:(X)$) -- ($(P)!0.2cm!90:(X)$);

    % Angle Theta at X
    \pic [draw, "$\theta$", angle eccentricity=1.5, angle radius=0.6cm] {angle = O--X--Z};

    % Braces for distances (Flipped to the outside)
    
    % Brace for |x|
    % Moving from O to X, the left side (no mirror) is the "outer/lower" side
    \draw [decorate, decoration={brace, amplitude=10pt, raise=8pt}]
      (O) -- (X) node [black, midway, sloped, yshift=-30pt] {$|x|$};

    % Brace for |x-2t\xi|
    % Moving from Z to O, the left side (no mirror) is the "outer/upper" side
    % Alternatively, move O to Z with mirror to stay on the bottom
    %\draw [decorate, decoration={brace, amplitude=10pt, raise=8pt}]
    %  (X) -- (Z) node [black, midway, sloped, yshift=30pt] {$|x-2t\xi|$};
  \end{tikzpicture}
  \caption{\label{fig:CVext3-ref-diagram-1} The triangle with vertices $x$, $2t\xi$, and the origin, together with the geometric information required in the proof of~\eqref{eqn:CVext-3-uncond-refined}.  Here, $z(\lambda_0)$ is the point on the line joining the points $x$ and $2t\xi$ that is closest to the origin; elementary trigonometry shows that $|z(\lambda_0)| = |x| \sin \theta$, where $\theta$ is the angle at the vertex $x$.}
\end{figure}
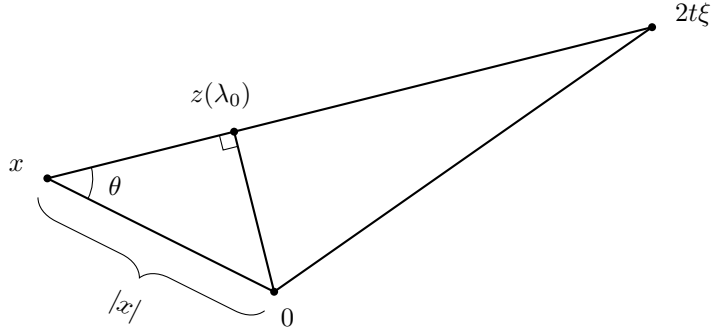
\begin{proof}
    We begin with the refined bound~\eqref{eqn:CVext-3-uncond-refined}.  We will focus on the case $a = 0$, since the other bounds are analogous.  Recall that
    \begin{equation*}
        \CVext^3(x,\xi,t) = \int_0^1 \int_0^{\lambda_1} \int_0^{\lambda} D_x^3 \Vext(\lambda x + 2(1-\lambda)t\xi) \;d\lambda d\lambda_2 d\lambda_1
    \end{equation*}
    and by the definition of $\Vext$,
    \begin{equation}\label{eqn:V-ext-cond}
        |D_x^3 \Vext(z)| \lesssim \max(|z|, t^\beta)^{-(3 + \mu)}
    \end{equation}
    Define $z = z(\lambda) = 2t\xi + \lambda (x-2t\xi)$, and let $\lambda_0 = \lambda_0(x,\xi,t)$ be the value of $\lambda$ that minimizes $|z(\lambda)|$.  In particular, $z(\lambda_0)$ is the orthogonal projection of the vector $0$ onto the line passing through the points $x$ and $2t\xi$, so we have that
    \begin{equation*}
        |z(\lambda)|^2 = |z(\lambda_0)|^2 + |x-2t\xi|^2(\lambda - \lambda_0)^2 
    \end{equation*}
    In particular, this shows that $|z(\lambda)| \geq |x-2t\xi||\lambda - \lambda_0|$, so
    \begin{equation}\label{eqn:D-x-3-CVext-angular-bound}
        |D_x^3 \Vext(z)| \lesssim \begin{cases}
            t^{-(3+\mu)\beta} & |\lambda - \lambda_0| \leq \frac{t^\beta}{|x - 2t\xi|}\\
            \left(|x-2t\xi||\lambda - \lambda_0|\right)^{-(3+\mu)} & |\lambda - \lambda_0| > \frac{t^\beta}{|x - 2t\xi|}
        \end{cases}
    \end{equation}
    Inserting this inequality in the definition of $\CVext^3$ and recalling the hypothesis that $|x - 2t\xi| \sim t^{1/2} 2^j$ now yields
    \begin{equation*}\begin{split}
        |\CVext^3(x,\xi,t)| \lesssim& \int_{\left\{|\lambda - \lambda_0| \leq \frac{t^\beta}{|x - 2t\xi|}\right\}} t^{-(3+\mu)\beta} \;d\lambda + \int_{\left\{|\lambda - \lambda_0| \geq \frac{t^\beta}{|x - 2t\xi|}\right\}} \left(|x-2t\xi||\lambda - \lambda_0|\right)^{-(3+\mu)}\;d\lambda\\
        \lesssim& t^{-(2+\mu)\beta -1/2} 2^{-j}
    \end{split}\end{equation*}
    as required by~\eqref{eqn:CVext-3-uncond-refined}.

    To obtain the angular improvement~\eqref{eqn:CVext-3-angular-refined}, we observe that
    \begin{equation*}
        |z(\lambda_0)| = |x|\sin \theta
    \end{equation*}
    Thus, we have that
    \begin{equation}\label{eqn:D-x-3-Vext-angular-bound}
        |D_x^3 \Vext(z)| \lesssim \begin{cases}
            (|x|\sin \theta)^{-(3+\mu)} & |\lambda - \lambda_0| \leq \frac{|x|\sin \theta}{|x - 2t\xi|}\\
            \left(|x-2t\xi||\lambda - \lambda_0|\right)^{-(3+\mu)} & |\lambda - \lambda_0| > \frac{|x|\sin \theta}{|x - 2t\xi|}
        \end{cases}
    \end{equation}
    and inserting this into the definition of $\CVext^3$ gives~\eqref{eqn:CVext-3-angular-refined}.
\end{proof}

To fully take advantage of~\eqref{eqn:CVext-3-angular-refined}, we need the following lemma:
\begin{lemma}\label[lemma]{lem:freq-space-vol-bound}
    For fixed $x$, the set of all $\xi$ such that
    \begin{equation}
        |x - 2t\xi| \in [t^{1/2} 2^j, t^{1/2} 2^{j+2}]
    \end{equation}
    and
    \begin{equation}
        \theta \in [2^{\ell-1}, 2^{\ell+1}]
    \end{equation}
    (with $\theta$ defined as in~\Cref{lem:CVext-3-exterior-bound-refined}) has volume $O(t^{-3/2} 2^{3j} 2^{2\ell})$.
\end{lemma}
\begin{proof}
    If we recenter the origin at $x$, then $\theta$ becomes the polar angle and $r = x - 2t\xi$ becomes a radial coordinate; see~\Cref{fig:spherical-coords}.  Denoting the azimuthal angle by $\phi$, we see that the volume of the set of all points $z$ such that $r = |x - z| \in [t^{1/2} 2^j,t^{1/2} 2^{j+2}]$ and $\theta = \arccos \frac{x \cdot (z-x)}{|x||x-z|} \in [2^{\ell-1}, 2^{\ell+1}]$ is given by
    \begin{equation*}
        \int_0^{2\pi} \int_{2^{\ell-1}}^{2^{\ell+1}} \int_{t^{1/2}2^j}^{t^{1/2}2^{j+2}} \!\!\!\!r^2 \sin \theta \;dr d\theta d\phi = O(t^{3/2} 2^{3j} 2^{2\ell})
    \end{equation*}
    Since $z = 2t\xi$, dividing by the Jacobian factor $(2t)^3$ gives the desired frequency-space volume.
\end{proof}
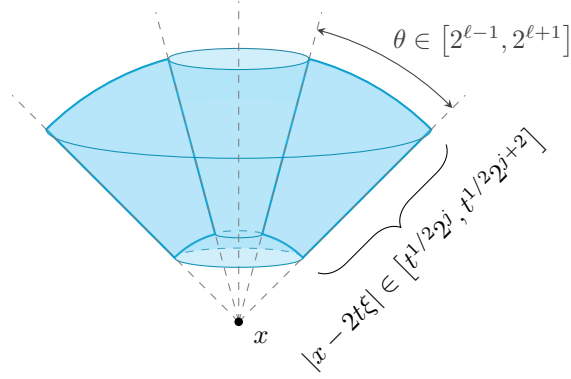
\begin{figure}
    \begin{center}
\begin{tikzpicture}[scale=0.8]
    % --- CONFIGURATION ---
    \def\r{1.5}       % Inner radius (represents 1 * t^{1/2} 2^j)
    \def\R{4.5}       % Outer radius (represents 4 * t^{1/2} 2^j)
    \def\angMin{15}   % Inner angle (represents 1/2 * 2^n)
    \def\angMax{45}   % Outer angle (represents 2 * 2^n)
    \def\tilt{0.15}   % Perspective tilt factor for the 3D ellipses
    
    % --- COORDINATES ---
    \coordinate (X) at (0, 0);    % Point x
    
    % --- 3D BACKGROUND (Translucent back surfaces and edges) ---
    % Drawn before the 2D section so they sit behind it

    % Outer spherical cap (back)
    \fill[cyan!40, opacity=0.25]
        (90-\angMax:\R) arc (90-\angMax:90-\angMin:\R)
        arc (0:180: {\R*sin(\angMin)} and {\tilt*\R*sin(\angMin)})
        arc (90+\angMin:90+\angMax:\R)
        arc (180:0: {\R*sin(\angMax)} and {\tilt*\R*sin(\angMax)})
        -- cycle;

    % Outer conical wall (back)
    \fill[cyan!50, opacity=0.15]
        (90-\angMax:\r) -- (90-\angMax:\R)
        arc (0:180: {\R*sin(\angMax)} and {\tilt*\R*sin(\angMax)})
        -- (90+\angMax:\r)
        arc (180:0: {\r*sin(\angMax)} and {\tilt*\r*sin(\angMax)})
        -- cycle;

    % Inner spherical cap (back)
    \fill[cyan!40, opacity=0.25]
        (90-\angMax:\r) arc (90-\angMax:90-\angMin:\r)
        arc (0:180: {\r*sin(\angMin)} and {\tilt*\r*sin(\angMin)})
        arc (90+\angMin:90+\angMax:\r)
        arc (180:0: {\r*sin(\angMax)} and {\tilt*\r*sin(\angMax)})
        -- cycle;

    % Inner conical wall (back)
    \fill[cyan!50, opacity=0.15]
        (90-\angMin:\r) -- (90-\angMin:\R)
        arc (0:180: {\R*sin(\angMin)} and {\tilt*\R*sin(\angMin)})
        -- (90+\angMin:\r)
        arc (180:0: {\r*sin(\angMin)} and {\tilt*\r*sin(\angMin)})
        -- cycle;

    % Back edges
    \draw[dashed, cyan!60!black, thin] (90-\angMin:\r) arc (0:180: {\r*sin(\angMin)} and {\tilt*\r*sin(\angMin)});
    \draw[dashed, cyan!60!black, thin] (90-\angMax:\r) arc (0:180: {\r*sin(\angMax)} and {\tilt*\r*sin(\angMax)});
    \draw[cyan!60!black, thin] (90-\angMin:\R) arc (0:180: {\R*sin(\angMin)} and {\tilt*\R*sin(\angMin)});
    %\draw[cyan!60!black, thin] (90-\angMax:\R) arc (0:180: {\R*sin(\angMax)} and {\tilt*\R*sin(\angMax)});

    % --- DRAW THE TARGET REGIONS (Opaque 2D Slice) ---
    % Right side patch
    \filldraw[fill=cyan!20, draw=cyan!80!black, thick]
        (90-\angMin:\r) arc (90-\angMin:90-\angMax:\r)
        -- (90-\angMax:\R) arc (90-\angMax:90-\angMin:\R)
        -- cycle;
        
    % Left side patch
    \filldraw[fill=cyan!20, draw=cyan!80!black, thick]
        (90+\angMin:\r) arc (90+\angMin:90+\angMax:\r)
        -- (90+\angMax:\R) arc (90+\angMax:90+\angMin:\R)
        -- cycle;

    % --- 3D FOREGROUND (Translucent surfaces) ---
    % Drawn over the 2D slice to embed it in "glass"

    % Inner conical wall (front)
    \fill[cyan!50, opacity=0.15]
        (90-\angMin:\r) -- (90-\angMin:\R)
        arc (0:-180: {\R*sin(\angMin)} and {\tilt*\R*sin(\angMin)})
        -- (90+\angMin:\r)
        arc (-180:0: {\r*sin(\angMin)} and {\tilt*\r*sin(\angMin)})
        -- cycle;
        
    % Outer conical wall (front)
    \fill[cyan!50, opacity=0.15]
        (90-\angMax:\r) -- (90-\angMax:\R)
        arc (0:-180: {\R*sin(\angMax)} and {\tilt*\R*sin(\angMax)})
        -- (90+\angMax:\r)
        arc (-180:0: {\r*sin(\angMax)} and {\tilt*\r*sin(\angMax)})
        -- cycle;

    % Inner spherical cap (front)
    \fill[cyan!40, opacity=0.25]
        (90-\angMax:\r) arc (90-\angMax:90-\angMin:\r)
        arc (0:-180: {\r*sin(\angMin)} and {\tilt*\r*sin(\angMin)})
        arc (90+\angMin:90+\angMax:\r)
        arc (-180:0: {\r*sin(\angMax)} and {\tilt*\r*sin(\angMax)})
        -- cycle;

    % Outer spherical cap (front)
    \fill[cyan!40, opacity=0.25]
        (90-\angMax:\R) arc (90-\angMax:90-\angMin:\R)
        arc (0:-180: {\R*sin(\angMin)} and {\tilt*\R*sin(\angMin)})
        arc (90+\angMin:90+\angMax:\R)
        arc (-180:0: {\R*sin(\angMax)} and {\tilt*\R*sin(\angMax)})
        -- cycle;

    % --- 3D FOREGROUND EDGES (Front halves of ellipses) ---
    \draw[cyan!70!black, thin] (90-\angMin:\r) arc (0:-180: {\r*sin(\angMin)} and {\tilt*\r*sin(\angMin)});
    \draw[cyan!70!black, thin] (90-\angMax:\r) arc (0:-180: {\r*sin(\angMax)} and {\tilt*\r*sin(\angMax)});
    \draw[cyan!70!black, thin] (90-\angMin:\R) arc (0:-180: {\R*sin(\angMin)} and {\tilt*\R*sin(\angMin)});
    \draw[cyan!70!black, thin] (90-\angMax:\R) arc (0:-180: {\R*sin(\angMax)} and {\tilt*\R*sin(\angMax)});

    % --- AXES & EXTENSIONS ---
    % Dotted guide lines showing the full conical boundaries
    \draw[dashed, gray] (X) -- (90-\angMin:\R + 0.8);
    \draw[dashed, gray] (X) -- (90-\angMax:\R + 0.8);
    \draw[dashed, gray] (X) -- (90+\angMin:\R + 0.8);
    \draw[dashed, gray] (X) -- (90+\angMax:\R + 0.8);

    % Label the thickness of the shell
    \draw [decorate, decoration={brace, mirror, amplitude=10pt, raise=8pt}]
      (90-\angMax:\r) -- (90-\angMax:\R) node [black, midway, sloped, yshift=-30pt] {$|x-2t\xi| \in  \left[t^{1/2} 2^{j},t^{1/2} 2^{j+2}\right]$};
    
    % Center axis (continuation of x)
    \draw[dashed, gray] (X) -- (0, \R + 0.8);

    % --- VECTORS ---
    \fill (X) circle (2pt) node[below right, xshift=2pt] {$x$};

    % --- ANGLE ANNOTATIONS ---
    % Draw angle arc for theta
    \draw[<->, >=stealth, darkgray] (3.535, 3.535) arc (45:75:5) node[midway, above right, inner sep=1pt] {$\theta \in \left[2^{\ell-1},2^{\ell+1}\right]$};

\end{tikzpicture}
\end{center}
\caption{\label{fig:spherical-coords} The region described in~\Cref{lem:freq-space-vol-bound}.}
\end{figure}

Returning to the interior region $|x| \leq t^{1/2}2^{j+10}$, we integrate by parts twice using $\mathcal{L}_0$ and introduce a further angular decomposition to write
\begin{equation*}
    \trmInst^j(x,t) = \trmInst^{j,< L}(x,t) + \sum_{\ell = L}^\infty \trmInst^{j,\ell}(x,t)
\end{equation*}
where we define
\begin{equation*}
    \trmInst^{j,< L} := \int e^{ix\cdot \xi - it\xi^2} F(2^{-L} |\theta| \leq 1) (\mathcal{L}_0^*)^2 \left[ \chi_j(t^{1/2}|\xi -\xi_\textup{stat}|) D_\xi^3 \cdot \left(\CVext^3(x,\xi,t) \hat{f}\right)\right] \;d\xi
\end{equation*}
and
\begin{equation*}
    \trmInst^{j,\ell} := \int e^{ix\cdot \xi - it\xi^2} \chi_\ell(\theta) (\mathcal{L}_0^*)^2 \left[ \chi_j(t^{1/2}|\xi -\xi_\textup{stat}|) D_\xi^3 \cdot \left(\CVext^3(x,\xi,t) \hat{f}\right)\right] \;d\xi
\end{equation*}
If we choose $L = L(x,t)$ so that $2^L |x| \sim t^\beta$, then using~\Cref{lem:freq-space-vol-bound} and~\eqref{eqn:CVext-3-uncond-refined} from~\Cref{lem:CVext-3-exterior-bound-refined}, we obtain the pointwise bound
\begin{equation*}\begin{split}
    |\trmInst^{j,< N(x,t)}(x,t)| \lesssim& t^{-9/4 - (2+\mu)\beta} 2^{-3/2j} 2^L t^{3\rho}\left(t 2^{-{2j}} + t^{2\rho}\right)\\
    \lesssim& t^{-9/4 - (1+\mu)\beta} 2^{-3/2j} t^{3\rho}\left(t 2^{-{2j}} + t^{2\rho}\right) |x|^{-1}
\end{split}\end{equation*}
On the other hand, for $\ell \geq L(x,t)$, we can instead use~\eqref{eqn:CVext-3-angular-refined} to obtain
\begin{equation*}
    |\trmInst^{j,\ell}(x,t)| \lesssim t^{-9/4} 2^{-3/2j} 2^{-(1+\mu)\ell} t^{3\rho}\left(t 2^{-{2j}} + t^{2\rho}\right)|x|^{-(2+\mu)}
\end{equation*}
Summing over $\ell \geq L$ and recalling that $2^L \sim |x|^{-1} t^\beta$ now gives
\begin{equation*}
    |\trmInst^j(x,t)| \lesssim t^{-9/4 - (1+\mu)\beta} 2^{-3/2j} t^{3\rho}\left(t 2^{-{2j}} + t^{2\rho}\right) |x|^{-1}
\end{equation*}
Taking the $L^2$ integral over the region $|x| \leq t^{1/2} 2^{j+10}$ now gives
\begin{equation}\label{eqn:trmInst-int}
    \lVert \trmInst^j(x,t) \rVert_{L^2(|x| \leq t^{1/2} 2^{j+10})} \lesssim t^{-2 - (1+\mu)\beta} 2^{-j} t^{3\rho}\left(t^{2\rho} + t 2^{-2j}\right)
\end{equation}
which is integrable in $t$ after summing over $j \geq 0$ based on our assumptions on $\rho$ and $\beta$.

\subsection{Existence of $\Omega_\ext^{F_\alpha \gamma F_\alpha}$}\label{sec:ext-F-gamma-F-WO}

The existence of $\Omega_\ext^{F_\alpha \gamma F_\alpha}$ is a straightforward application of Cook's method.  We see that
\begin{equation*}
    \Omega_\ext^{F_\alpha \gamma F_\alpha} = U_\ext(0,1) F \gamma F U(1,0) + \int_1^\infty \frac{d}{dt} \left(U_\ext(0,t) F_\alpha \gamma F_\alpha U(t,0)\right)\;dt
\end{equation*}
provided that the integral term converges in $L^2$.  Expanding the derivative, we see that
\begin{align*}
    \frac{d}{dt} \left(U_\ext(0,t) F_\alpha \gamma F_\alpha U(t,0)\right) =& U_\ext(0,t) (-i\Delta + iV_\ext(x,t))  F_\alpha \gamma F_\alpha U(t,0)\\
    &+ U_\ext(0,t) (\frac{d}{dt} F_\alpha) \gamma F_\alpha U(t,0) + U_\ext(0,t) F_\alpha \gamma (\frac{d}{dt} F_\alpha) U(t,0)\\
    &+ U_\ext(0,t)   F_\alpha \gamma F_\alpha (i\Delta - iV(x,t)) U(t,0)
\end{align*}
Since we are assuming that $\beta < \alpha$, $F_\alpha V_\ext = F_\alpha V$, so this becomes
\begin{subequations}\begin{align}
    \frac{d}{dt} \left(U_\ext(0,t) F_\alpha \gamma F_\alpha U(t,0)\right) =& U_\ext [-i\Delta, F_\alpha \gamma F_\alpha] U(t,0)\label{eqn:ext-rad-mom-WO-lap-comm}\\
    &+ U_\ext \left((\frac{d}{dt} F_\alpha) \gamma F_\alpha + F_\alpha \gamma (\frac{d}{dt} F_\alpha) \right)U(t,0)\label{eqn:ext-rad-mom-WO-time-deriv}\\
    &+ U_\ext  F_\alpha [iV,\gamma] F_\alpha U(t,0)\label{eqn:ext-rad-mom-WO-pot-comm}
\end{align}\end{subequations}
For~\eqref{eqn:ext-rad-mom-WO-pot-comm}, we immediately see that
\begin{equation*}
    \lVert \eqref{eqn:ext-rad-mom-WO-pot-comm} \rVert_{L^2 \to L^2} \lesssim t^{-(1+\mu) \alpha}
\end{equation*}
is integrable in time, so it only remains to consider the other two terms.  For~\eqref{eqn:ext-rad-mom-WO-lap-comm}, we observe that
\begin{equation*}
    [-i\Delta, F_\alpha \gamma F_\alpha] = F_\alpha [-i\Delta, \gamma] F_\alpha + [-i\Delta, F_\alpha] \gamma F_\alpha + F_\alpha \gamma [-i\Delta, F_\alpha] 
\end{equation*}
For the leading term, we have by~\eqref{eqn:wls-symb-comm} that
\begin{equation*}
    F_\alpha [-i\Delta, \gamma] F_\alpha = - 2 F_\alpha \nabla_x \cdot \frac{1}{|x|}\left(I - \frac{x}{|x|} \otimes \frac{x}{|x|}\right) \nabla_x F_\alpha
\end{equation*}
For the remaining two terms, we observe that
\begin{equation}\label{eqn:F-alpha-lap-comm-id}\begin{split}
    [-i\Delta, F_\alpha]    =& t^{-\alpha} \left(F_\alpha' \gamma + \gamma F_\alpha'\right)\\
                            =& 2t^{-\alpha} F'_\alpha \gamma + t^{-2\alpha} F''_\alpha\\
                            =& 2t^{-\alpha} \gamma F'_\alpha - t^{-2\alpha} F''_\alpha\\
\end{split}\end{equation}
where we adopt the convention that
\begin{equation*}
    F'_\alpha(r) := F'\left(\frac{r}{t^\alpha} \geq 1\right)
\end{equation*}
and similarly for higher derivatives of $F_\alpha$.  In particular, symmetrizing using~\Cref{lem:three-term-sym} gives
\begin{equation*}\begin{split}
    \eqref{eqn:ext-rad-mom-WO-lap-comm} =& -2 U_\ext(0,t) F_\alpha \nabla_x \cdot S(x) \nabla_x F_\alpha U(t,0)\\
    &+ 2t^{-\alpha} U_\ext(0,t)  \left(F'_\alpha \gamma^2 F_\alpha + F_\alpha \gamma^2 F'_\alpha\right) U(t,0)\\
    &+ t^{-2\alpha} U_\ext(0,t) \left(F''_\alpha \gamma F_\alpha - F_\alpha \gamma F''_\alpha\right) U(t,0)\\
    =& -2 U_\ext(0,t) F_\alpha \nabla_x \cdot S(x) \nabla_x F_\alpha U(t,0)\\
    &+ 4t^{-\alpha}U_\ext(0,t)  \sqrt{F_\alpha F'_\alpha} \gamma^2 \sqrt{F_\alpha F'_\alpha} U(t,0) + R(t)
\end{split}\end{equation*}
where $R(t)$ satisfies
\begin{equation*}
    \lVert R(t) \rVert_{L^2 \to L^2} \lesssim t^{-3\alpha}
\end{equation*}
In particular, the remainder $R(t)$ is integrable, so it only remains to handle the main term.  For this purpose, we use the following bound:
\begin{proposition}
    For any $\phi \in L^2$ and any $u_0 \in H^1$ such that $u(t) = U(t,0) u_0$ satisfies the finite energy condition
    \begin{equation*}
        \lVert u(x,t) \rVert_{L^\infty_t H^1_x} < \infty
    \end{equation*}
    we have that
    \begin{equation}\label{eqn:crooked-PRES-conc-1}\begin{split}
        \int_1^T \langle U_\ext(t,0) \phi,  \left(4t^{-\alpha}\sqrt{F_\alpha F'_\alpha} \gamma^2 \sqrt{F_\alpha F'_\alpha} -2 F_\alpha \nabla_x \cdot S(x) \nabla_xF_\alpha\right) u(t) \rangle\;dt \lesssim \lVert \phi \rVert_{L^2} \lVert u(x,t) \rVert_{L^\infty_t H^1_x}
    \end{split}\end{equation}
    uniformly in $T$.  In particular, 
    \begin{equation}\label{eqn:crooked-PRES-conc-2}
        \int_1^\infty \langle U_\ext(t,0) \phi, \left(4t^{-\alpha}\sqrt{F_\alpha F'_\alpha} \gamma^2 \sqrt{F_\alpha F'_\alpha} -2 F_\alpha \nabla_x \cdot S(x) \nabla_xF_\alpha\right) u(t) \rangle_{L^2_x}\;dt  \lesssim \lVert \phi \rVert_{L^2} \lVert u(x,t) \rVert_{L^\infty_t H^1_x}
    \end{equation}
\end{proposition}
\begin{proof}
    To lighten the notation, for an operator $A$, we denote
    \begin{equation*}
        \langle A \rangle_{\phi,u,t} := \langle U_\ext(t,0) \phi, A u(t) \rangle_{L^2_x}
    \end{equation*}
    Observe that we have the Heisenberg derivative-like formula
    \begin{equation*}
        \frac{d}{dt} \langle A(t) \rangle_{\phi,u,t} = \langle (-i\Delta + i \Vext)A(t) \rangle_{\phi,u,t} + \langle A(t)(i\Delta - iV) \rangle_{\phi,u,t} + \langle \frac{d}{dt} A(t) \rangle_{\phi,u,t}
    \end{equation*}
    In particular, for $A(t) = F_\alpha \gamma F_\alpha$, the fact that $F_\alpha V = F_\alpha V_\ext$ implies that
    \begin{equation}\label{eqn:crooked-PRES}
        \frac{d}{dt} \langle F_\alpha \gamma F_\alpha \rangle_{\phi, u, t} = \langle [-i\Delta, F_\alpha \gamma F_\alpha] \rangle_{\phi, u, t} + \langle F_\alpha [iV,\gamma] F_\alpha \rangle_{\phi, u, t} + \langle \frac{d}{dt} F_\alpha \gamma F_\alpha + F_\alpha \gamma \frac{d}{dt} F_\alpha \rangle_{\phi, u, t}
    \end{equation}
    By repeating the arguments for~\eqref{eqn:ext-rad-mom-WO-lap-comm} and~\eqref{eqn:ext-rad-mom-WO-pot-comm}, we see that
    \begin{equation*}\begin{split}
        \langle [-i\Delta, F_\alpha \gamma F_\alpha] \rangle_{\phi, u, t} + \langle F_\alpha [iV,\gamma] F_\alpha \rangle_{\phi, u, t} =& -2 \langle F_\alpha \nabla_x \cdot S(x) \nabla_x F_\alpha \rangle_{\phi, u, t}\\
        &+4t^{-\alpha}\langle \sqrt{F_\alpha F'_\alpha} \gamma^2 \sqrt{F_\alpha F'_\alpha} \rangle_{\phi, u, t} \\
        &+ O(t^{-(1+\mu)\alpha} \lVert \phi \rVert_{L^2_x}\lVert u \rVert_{L^\infty_t L^2_x})
    \end{split}\end{equation*}
    For the term involving time derivatives, we see that
    \begin{equation*}
        \frac{d}{dt} F_\alpha = \frac{|x|}{t^{\alpha +1}} F'_\alpha
    \end{equation*}
    so
    \begin{equation*}\begin{split}
        \langle \frac{d}{dt} F_\alpha \gamma F_\alpha + F_\alpha \gamma \frac{d}{dt} F_\alpha \rangle_{\phi, u, t} =& \frac{1}{t}\langle \frac{|x|}{t^\alpha} F'_\alpha \gamma F_\alpha + F_\alpha \gamma F'_\alpha\frac{|x|}{t^\alpha}\rangle_{\phi,u,t}\\
        =& \frac{1}{t} \langle \frac{|x|}{t^\alpha}(F_\alpha' \gamma F_\alpha + F_\alpha \gamma F_\alpha')\rangle_{\phi,u,t}
        + \frac{1}{t^{1+\alpha}}\langle F_\alpha F_\alpha' \rangle_{\phi,u,t}\\
        =& \frac{2}{t} \langle \frac{|x|}{t^\alpha}\sqrt{F_\alpha F_\alpha'} \gamma \sqrt{F_\alpha F_\alpha'}\rangle_{\phi,u,t} + O(t^{-1-\alpha}\lVert \phi \rVert_{L^2} \lVert u \rVert_{L^\infty_t H^1_x})
    \end{split}\end{equation*}
    Now, by~\Cref{cor:F-gamma-F-PRES}, we see that $t^{-\alpha/2} \gamma \sqrt{F_\alpha F_\alpha'} u \in L^2_{t,x}$ for finite energy solutions $u$, so for any $T > 1$
    \begin{equation*}\begin{split}
        \int_1^T \left|\langle \frac{d}{dt} F_\alpha \gamma F_\alpha + F_\alpha \gamma \frac{d}{dt} F_\alpha \rangle_{\phi, u, t}\right| \;dt \lesssim&
        \int_1^T t^{\alpha/2 - 1} \lVert \phi \rVert_{L^2_x} \lVert t^{-\alpha/2} \gamma \sqrt{F_\alpha F'_\alpha} u \rVert_{L^2_x} \;dt + \lVert \phi \rVert_{L^2_x} \lVert u \rVert_{L^\infty_t L^2_x}\\
        \lesssim& \lVert \phi \rVert_{L^2} \lVert u \rVert_{L^\infty_t H^1_x}
    \end{split}\end{equation*}
    with an implicit constant independent of $T$.  Thus, integrating in~\eqref{eqn:crooked-PRES} shows that
    \begin{equation*}\begin{split}
        \int_1^T \langle -2 F_\alpha \nabla_x \cdot S(x) \nabla_x F_\alpha +  4t^{-\alpha}\sqrt{F_\alpha F'_\alpha} \gamma^2 \sqrt{F_\alpha F'_\alpha} \rangle_{\phi,u,t}\;dt =& \langle F_\alpha \gamma F_\alpha \rangle_{\phi,u,T} - \langle F_\alpha \gamma F_\alpha \rangle_{\phi,u,1}\\
        &+ O(\lVert \phi \rVert_{L^2} \lVert u \rVert_{L^\infty_t H^1_x}) 
    \end{split}\end{equation*}
    uniformly in $T$, which is~\eqref{eqn:crooked-PRES-conc-1}.  Furthermore, the $O(\lVert \phi \rVert_{L^2} \lVert u \rVert_{L^\infty_t H^1_x})$ term converges as $T \to \infty$, which gives~\eqref{eqn:crooked-PRES-conc-2}.
\end{proof}

\section{Sublinear spreading for weakly bound states}\label{sec:sublin-spread}

We will now prove~\Cref{thm:uwl-sublin-spreading} by modifying the argument from~\Cref{sec:better-rate}.  The main idea is that we can replace $|x|$ with a truncated weight and obtain estimates that are uniform in the truncation.  More specifically, for $R > 0$, define the function $G_R(r) = R G\left(\frac{r}{R}\right)$, where $G$ is a smoothing, increasing function satisfying
\begin{equation*}
    G(r) = \begin{cases}
        r & r \leq \frac{1}{2}\\
        1 & r \geq 2
    \end{cases}
\end{equation*}
For later use, we record the following lemma:
\begin{lemma}\label[lemma]{lem:GR-vanishing}
    For any $f \in L^2$,
    \begin{equation}
        \lim_{R \to \infty} \frac{1}{R}\lVert G_R(|x|) f(x) \rVert_{L^2} = 0 
    \end{equation}
\end{lemma}
\begin{proof}
    Since $\frac{1}{R} G_R$ is bounded by $1$ and decreases pointwise to $0$ as $R \to \infty$, this follows from the dominated convergence theorem.
\end{proof}

We now claim that the following propagation estimate holds for weakly bound states:
\begin{proposition}\label[proposition]{prop:uwl-GR-bound}
    For any weakly bound state $\uwl$ and $\alpha > \beta$, we have that
    \begin{equation}
        \langle F_\alpha G_R(|x|) F_\alpha \rangle_t = \langle F_\alpha G_R(|x|) F_\alpha \rangle_1 + o_{R \to \infty}(t) + o_{t \to \infty}(t)
    \end{equation}
\end{proposition}
Let us assume that~\Cref{prop:uwl-GR-bound} holds.  By~\Cref{lem:GR-vanishing}, we have that
\begin{equation*}
    \langle F_\alpha G_R(|x|) F_\alpha \rangle_1 = o_{R \to \infty}(R)
\end{equation*}
In particular, for any $\epsilon > 0$, there exist $R_0(\epsilon)$ and $T_0(\epsilon)$ such that for $R \geq R_0(\epsilon)$ and $t \geq T_0(\epsilon)$,
\begin{equation*}
    |\langle F_\alpha G_R(|x|) F_\alpha \rangle_t| \leq \epsilon^2(R + t)
\end{equation*}
By taking $R = \epsilon t$ for $t \geq T(\epsilon) = \max(\epsilon^{-1}R_0(\epsilon), \epsilon^{-\frac{2}{1-\alpha}}, T_0(\epsilon))$, we find that
\begin{equation*}
    |\langle F_\alpha G_{\epsilon t}(|x|) F_\alpha \rangle_t| \leq 2\epsilon^2 t
\end{equation*}
On the other hand, from the definition of $F_\alpha$ we see at once that
\begin{equation*}
    |\langle G_{\epsilon t}(|x|) (1-F_\alpha^2) \rangle_t| \lesssim t^{\alpha} \lesssim \epsilon^2 t 
\end{equation*}
Applying Markov's inequality now shows that
\begin{equation*}
    \left\lVert F\left(\frac{|x|}{\epsilon t} \geq 1 \right) F_\alpha \uwl(x,t) \right\rVert_{L^2_x}^2 \lesssim \frac{1}{\epsilon t}\langle F_\alpha G_{\epsilon t}(|x|) F_\alpha \rangle_t \lesssim \epsilon
\end{equation*}
In particular, the sublinear spreading of $\uwl$ will follow once we prove~\Cref{prop:uwl-GR-bound}.  Defining
\begin{equation}\label{eqn:gammaR-def}
    \gamma_R := \frac{1}{2}[-i\Delta, G_R(|x|)] = \frac{1}{2}\left( G_R'(|x|) \gamma + \gamma G_R'(|x|) \right)
\end{equation}
we see that
\begin{equation}\label{eqn:Heis-deriv-FGF}
    \frac{d}{dt} \langle F_\alpha G_R F_\alpha \rangle_t = 2\langle F_\alpha \gamma_R F_\alpha \rangle_t + 4 t^{-\alpha} \langle \sqrt{G_R(|x|) F_\alpha F'_\alpha} \gamma \sqrt{G_R(|x|) F_\alpha F'_\alpha} \rangle_t - 2\alpha \left\langle F_\alpha F'_\alpha \frac{|x|G_R}{t^{\alpha +1}} \right\rangle_t 
\end{equation}
Since $G_R(|x|) \lesssim |x|$ (uniformly in $R$), we see at once that
\begin{equation}\label{eqn:uwl-time-deriv-2}
    \left|\left\langle F_\alpha F'_\alpha \frac{|x|G_R}{t^{\alpha +1}} \right\rangle_t \right| \lesssim t^{\alpha - 1} = o_{t \to \infty}(t)
\end{equation}
Although the middle term has an operator norm that is $O(1)$ in time (or $O(R/t^\alpha)$ if we relax the requirement that our bounds be uniform in $R$), we can obtain a better time-integrated bound by using propagation estimates:
\begin{proposition}\label[proposition]{prop:single-gamma-PRES}
    For any time $t > 1$,
    \begin{equation}\label{eqn:single-gamma-PRES}
        \int_1^t s^{-\alpha} \langle \sqrt{G_R(|x|) F_\alpha F'_\alpha} \gamma \sqrt{G_R(|x|) F_\alpha F'_\alpha} \rangle_s \;ds \lesssim t^\alpha
    \end{equation}
    with an implicit constant independent of $R$.
\end{proposition}
\begin{proof}
    Let us define a function $K(r,t)$ by
    \begin{equation*}
        \left\{\begin{array}{c}
            K(0,t) = 0\\
            \partial_r K(r,t) = F\left( \frac{r}{t^\alpha}\right) F'\left( \frac{r}{t^\alpha}\right) \frac{G_R(r)}{t^\alpha}
        \end{array}\right.
    \end{equation*}
    Evidently, $K$ is positive and increasing in $r$ for fixed $t$. Additionally, since $|G_R(r)| \lesssim r$ and $F'\left( \frac{r}{t^\alpha}\right)$ is supported in the region $r \sim t^\alpha$, we have that
    \begin{equation*}
        \lim_{r \to \infty} K(r,t) \lesssim t^\alpha
    \end{equation*}
    Moreover,
    \begin{equation*}\begin{split}
        |\partial_t K(r,t)| =& \left|\partial_t \int_0^r \partial_r K(u,t)\;du\right|\\
                            =& \left| \partial_t \int_0^r F\left( \frac{u}{t^\alpha}\right) F'\left( \frac{u}{t^\alpha}\right) \frac{G_R(u)}{t^\alpha} \;du \right|\\
                            \lesssim& t^{\alpha - 1}
    \end{split}\end{equation*}
    Thus, we have that
    \begin{equation*}
        \frac{d}{dt} \left\langle K(|x|, t) \right\rangle_t = \langle \gamma \partial_r K(|x|,t) + \partial_r K(|x|,t) \gamma \rangle_t + O(t^{\alpha-1})
    \end{equation*}
    Now, 
    \begin{equation*}
        \langle \gamma \partial_r K(r,t) + \partial_r K(r,t) \gamma \rangle_t = 2t^{-\alpha} \langle \sqrt{G_R F_\alpha F'_\alpha} \gamma \sqrt{G_R F_\alpha F'_\alpha} \rangle_t
    \end{equation*}
    so integrating gives that
    \begin{equation*}\begin{split}
        \int_1^t \langle s^{-\alpha}\sqrt{G_R F_\alpha F'_\alpha} \gamma \sqrt{G_R F_\alpha F'_\alpha} \rangle_s\;ds =& \frac{1}{2} \left( \langle K(|x|,t) \rangle_t - \langle K(|x|, 1) \rangle_1 \right) + O(t^{\alpha})\\
        =& O(t^\alpha) \qedhere
    \end{split}\end{equation*}
\end{proof}
In particular, both~\eqref{eqn:uwl-time-deriv-2} and~\eqref{eqn:single-gamma-PRES} are compatible with~\Cref{prop:uwl-GR-bound}, so it only remains to obtain acceptable bounds for the leading term $2\langle F_\alpha \gamma_R F_\alpha \rangle_t$.  Since $\gamma_R f$ vanishes for $f$ supported outside a ball of radius $2R$, we see at once that
\begin{equation}\label{eqn:F-gamma-R-F-int}
    \langle F_\alpha \gamma_R F_\alpha \rangle_t = - \int_t^{(2R)^{1/\alpha}} \frac{d}{ds} \langle F_\alpha \gamma_R F_\alpha \rangle_s \;ds
\end{equation}
We compute that
\begin{equation*}\begin{split}
    \frac{d}{dt} \langle F_\alpha \gamma_R F_\alpha \rangle_t =& \langle F_\alpha [-i\Delta, \gamma_R] F_\alpha \rangle_t+ \langle [-i\Delta, F_\alpha] \gamma_R F_\alpha + F_\alpha \gamma_R [-i\Delta, F_\alpha] \rangle_t\\
    &+ \langle F_\alpha [iV, \gamma_R] F_\alpha \rangle_t - \alpha \left\langle \frac{|x|}{t^{\alpha+1}} F'_\alpha \gamma_R F_\alpha + F_\alpha \gamma_R F'_\alpha \frac{|x|}{t^{\alpha+1}}\right \rangle_t 
\end{split}
\end{equation*}
For the commutator with the potential,~\eqref{eqn:gammaR-def} immediately shows that
\begin{equation*}
    \langle F_\alpha [iV, \gamma_R] F_\alpha \rangle_t = O(t^{-(1+\mu)\alpha})
\end{equation*}
For the terms involving $[-i\Delta, F_\alpha]$, we symmetrize with~\Cref{lem:three-term-sym} to find that
\begin{equation*}\begin{split}
    [-i\Delta, F_\alpha] \gamma_R F_\alpha + F_\alpha \gamma_R [-i\Delta, F_\alpha] = 2t^{-\alpha}\sqrt{F_\alpha F'_\alpha} (\gamma \gamma_R + \gamma_R \gamma) \sqrt{F_\alpha F'_\alpha} + O(t^{-3\alpha})
\end{split}
\end{equation*}
Using~\eqref{eqn:gammaR-def}, we further see that
\begin{equation*}\begin{split}
    2\left(\gamma \gamma_R + \gamma_R \gamma\right) =& 2 \gamma G_R' \gamma + \gamma^2 G_R' + G_R' \gamma^2\\
    =& 4 \gamma G_R' \gamma + \gamma [\gamma, G_R'] + [G_R', \gamma]\gamma\\
    =& 4 \gamma G_R' \gamma + [\gamma, [\gamma, G_R']]\\
    =& 4 \gamma G_R' \gamma - G_R'''
\end{split}\end{equation*}
so,
\begin{equation*}
    2t^{-\alpha}\sqrt{F_\alpha F'_\alpha} (\gamma \gamma_R + \gamma_R \gamma) \sqrt{F_\alpha F'_\alpha} = 4 t^{-\alpha} \sqrt{F_\alpha F'_\alpha} \gamma G_R' \gamma \sqrt{F_\alpha F'_\alpha} - t^{-\alpha} F_\alpha F'_\alpha G_R'''
\end{equation*}
Since $\lVert G_R''' \rVert_{L^\infty} = O(R^{-2})$ and the supports of $F_\alpha$ and $G_R'''$ only intersect for $t^\alpha \lesssim R$, this gives that
\begin{equation*}
    [-i\Delta, F_\alpha] \gamma_R F_\alpha + F_\alpha \gamma_R [-i\Delta, F_\alpha] = 4 t^{-\alpha} \sqrt{F_\alpha F'_\alpha} \gamma G_R' \gamma \sqrt{F_\alpha F'_\alpha} + O(t^{-3\alpha})
\end{equation*}
For the term involving time derivatives of $F_\alpha$, we use Cauchy-Schwarz and the Cauchy inequality to find that
\begin{equation*}\begin{split}
    \left| \alpha \left\langle \frac{|x|}{t^{\alpha+1}} F'_\alpha \gamma_R F_\alpha + F_\alpha \gamma_R F'_\alpha \frac{|x|}{t^{\alpha+1}}\right \rangle_t \right| \leq& \frac{\alpha}{t^{\alpha + 1}} |\langle  \sqrt{F_\alpha F'_\alpha} \gamma G_R' \sqrt{F_\alpha F'_\alpha}|x| \rangle_t |\\
    &+  \frac{\alpha}{t^{\alpha + 1}} |\langle  |x|\sqrt{F_\alpha F'_\alpha}  G_R'\gamma \sqrt{F_\alpha F'_\alpha} \rangle_t|\\
    \lesssim& \frac{1}{t} \lVert \sqrt{G_R'} \gamma \sqrt{F_\alpha F'_\alpha} \uwl(x,t)\rVert_{L^2} \lVert \sqrt{F_\alpha F'_{\alpha}} \uwl(x,t) \rVert_{L^2}\\
    \leq& 2 t^{-\alpha} \langle \sqrt{F_\alpha F'_\alpha} \gamma G_R' \gamma \sqrt{F_\alpha F'_\alpha} \rangle_t + O(t^{\alpha - 2})
\end{split}\end{equation*}
The most difficult term is the one involving the commutator of the Laplacian with $\gamma_R$.  We compute that
\begin{equation*}\begin{split}
    2 F_\alpha [-i\Delta, \gamma_R] F_\alpha =& F_\alpha [-i\Delta, \gamma] G_R' F_\alpha + F_\alpha G_R' [-i\Delta, \gamma] F_\alpha + F_\alpha \gamma [-i\Delta, G_R'] F_\alpha + F_\alpha [-i\Delta, G_R'] \gamma F_\alpha
\end{split}\end{equation*}
From~\eqref{eqn:wls-symb-comm}, we have that 
\begin{equation*}\begin{split}
    F_\alpha [-i\Delta, \gamma] G_R' F_\alpha + F_\alpha G_R' [-i\Delta, \gamma] F_\alpha =& -2F_\alpha \left(\nabla_x \cdot S(x) \nabla_x G_R' + G_R' \nabla_x \cdot S(x) \nabla_x\right)F_\alpha\\
    =& -4 F_\alpha \nabla_x \cdot (G_R'(|x|) S(x)) \nabla_x F_\alpha + F_\alpha[G_R', \nabla_x] \cdot S(x) \nabla_x F_\alpha\\
    &+ F_\alpha \nabla_x \cdot S(x) [\nabla_x, G_R'(|x|)] F_\alpha\\
    =& -4 F_\alpha \nabla_x \cdot (G_R'(|x|) S(x)) \nabla_x F_\alpha
\end{split}\end{equation*}
where the last line follows from the fact that $\nabla_x G_R'(|x|)$ is purely radial and $S(x)$ projects onto transverse directions.  In particular, these terms have a favorable sign in~\eqref{eqn:F-gamma-R-F-int}, so we will focus on the other two terms.  Here, we see that
\begin{equation*}\begin{split}
    \gamma[-i\Delta, G_R'] + [-i\Delta, G_R']\gamma =& \gamma^2 G_R'' + 2 \gamma G_R'' \gamma + G_R'' \gamma^2\\
    =& 4 \gamma G_R'' \gamma - G_R^{(4)}
\end{split}\end{equation*}
Since the supports of $G_R$ and $F_\alpha$ only overlap for $t^\alpha < 2R$ and $\lVert G_R^{(4)} \rVert_{L^\infty_x} \sim R^{-3}$, we see that
\begin{equation*}\begin{split}
    \int_t^\infty \langle F_\alpha G_R^{(4)} F_\alpha \rangle_s\;ds =& \int_t^{(2R)^{1/\alpha}} O(R^{-3}) \;ds\\
                                                \lesssim& O(R^{1/\alpha - 3})\\
                                                =& o_{R \to \infty}(1)
\end{split}\end{equation*}
On the other hand, the leading term $4 \langle F_\alpha \gamma G_R'' \gamma F_\alpha \rangle_t$ has an unfavorable sign and is too large to be handled perturbatively.  Instead, we will control this term using a combination of propagation estimates and the decay of $\langle F_\alpha \gamma F_\alpha \rangle_t$.  To begin, we reorder the operators to obtain
\begin{equation*}\begin{split}
    2F_\alpha \gamma G_R'' \gamma F_\alpha  = 2 \gamma (F_\alpha G_R'' F_\alpha) \gamma + [\gamma, G_R''[\gamma, F_\alpha]]F_\alpha + F_\alpha [\gamma, G_R''[\gamma, F_\alpha]]
\end{split}\end{equation*}
The double commutator terms are supported in the region $R \sim t^\alpha$, and have size $O(R^{-3})$.  Thus, they can be treated perturbatively in the same manner as $\langle F_\alpha G^{(4)}_R F_\alpha \rangle_t$.  For the remaining term, let $\tilde{F}_R = \tilde{F}\left( \frac{|x|}{R} \geq 1\right)$ be a smooth cut-off function such that $\tilde{F}(r\geq 1)\tilde{F}'(r \geq 1) = 1$ on the support of $G''$. Reproducing, we find that
\begin{equation}\label{eqn:GR-reprod}\begin{split}
    \langle \gamma (F_\alpha G_R'' F_\alpha) \gamma \rangle_t =& \langle \sqrt{\tilde{F}_R \tilde{F}_R'} \gamma (F_\alpha G_R'' F_\alpha) \gamma \sqrt{\tilde{F}_R \tilde{F}_R'} \rangle_t\\
    \lesssim& \frac{1}{R}\left\lVert \gamma \sqrt{\tilde{F}_R \tilde{F}_R'} \uwl(x,t) \right\rVert_{L^2_x}^2\\
    \lesssim& \frac{1}{R} \left\langle \sqrt{\tilde{F}_R \tilde{F}_R'} \gamma^2 \sqrt{\tilde{F}_R \tilde{F}_R'} \right\rangle_t
\end{split}\end{equation}
This may be controlled using a propagation estimate.  Indeed, considering the observable $\tilde{F}_R \gamma \tilde{F}_R$ and arguing as in~\eqref{eqn:F-gamma-F-Heis}, we see that
\begin{equation}\label{eqn:FR-gamma-FR-Heis}\begin{split}
    \frac{d}{dt} \langle \tilde{F}_R \gamma \tilde{F}_R \rangle_t =& \langle \tilde{F}_R [-i\Delta, \gamma] \tilde{F}_R \rangle_t + \langle \tilde{F}_R [iV, \gamma] \tilde{F}_R \rangle_t\\
    &+ \langle [-i\Delta, \tilde{F}_R] \gamma \tilde{F}_R + \tilde{F}_R \gamma [-i\Delta, \tilde{F}_R]\rangle_t
\end{split}
\end{equation}
By the same reasoning used to obtain~\eqref{eqn:wls-symb-comm}, we see $\tilde{F}_R [-i\Delta, \gamma] \tilde{F}_R$ is positive semi-definite, so the first term in~\eqref{eqn:FR-gamma-FR-Heis} has a favorable sign.  Using the decay hypothesis for $V$, we see that
\begin{equation*}
    |\langle \tilde{F}_R [iV, \gamma] \tilde{F}_R \rangle_t| \lesssim R^{-1-\mu}
\end{equation*}
For the final term, we argue as in~\eqref{eqn:wls-bdy-comm} to find that
\begin{equation*}
    \langle [-i\Delta, \tilde{F}_R] \gamma \tilde{F}_R + \tilde{F}_R \gamma [-i\Delta, \tilde{F}_R]\rangle_t = \frac{4}{R} \langle \sqrt{\tilde{F}_R\tilde{F}'_R} \gamma^2 \sqrt{\tilde{F}_R\tilde{F}'_R} \rangle_t + O(R^{-3})
\end{equation*}
In particular, integrating in~\eqref{eqn:FR-gamma-FR-Heis} and recalling~\eqref{eqn:GR-reprod} yields
\begin{equation*}\begin{split}
    \int_t^\infty |\langle \gamma (F_\alpha G_R'' F_\alpha) \gamma \rangle_s|\;ds \leq& \int_0^{(2R)^{1/\alpha}} |\langle \gamma (F_\alpha G_R'' F_\alpha) \gamma \rangle_s| \;ds\\
    \lesssim& \int_0^{(2R)^{1/\alpha}} \frac{4}{R} \langle \sqrt{\tilde{F}_R\tilde{F}'_R} \gamma^2 \sqrt{\tilde{F}_R\tilde{F}'_R} \rangle_s \;ds\\
    \lesssim& \langle \tilde{F}_R \gamma \tilde{F}_R \rangle_{(2R)^{1/\alpha}} - \langle \tilde{F}_R \gamma \tilde{F}_R \rangle_{0} + O(R^{\frac{1}{\alpha} - 1 - \mu})
\end{split}\end{equation*}
Now, by dominated convergence $\langle \tilde{F}_R \gamma \tilde{F}_R \rangle_0 = o_{R \to \infty}(1)$.  On the other hand, by modifying the reasoning used to prove that $\langle F_\alpha \gamma F_\alpha \rangle_t \to 0$, we see that
\begin{equation*}
    \langle \tilde{F}_R \gamma \tilde{F}_R \rangle_{(2R)^{1/\alpha}} = \left.\left\langle \tilde{F}\left(\frac{|x|}{(2t)^\alpha} \geq 1\right) \gamma \tilde{F}\left(\frac{|x|}{(2t)^\alpha} \geq 1\right)\right\rangle_t \right|_{t = (2R)^{1/\alpha}} = o_{R \to \infty}(1)
\end{equation*}
which completes the proof of~\Cref{prop:uwl-GR-bound}.

\noindent{\textbf{Acknowledgements}: }A. S. was partially supported by NSF grant DMS-2205931. X.W. was supported by the ARC Australian Laureate Fellowship grant FL220100072.
\providecommand{\MR}{\relax\ifhmode\unskip\space\fi MR }
% \MRhref is called by the amsart/book/proc definition of \MR.
\providecommand{\MRhref}[2]{%
  \href{http://www.ams.org/mathscinet-getitem?mr=#1}{#2}
}
\providecommand{\href}[2]{#2}

\end{document}